\documentclass[12pt,twoside]{amsart}

\usepackage{amssymb,amsmath}
\theoremstyle{plain}
\newtheorem{theorem}{Theorem}[section]
\newtheorem{lemma}[theorem]{Lemma}

\theoremstyle{definition}
\newtheorem{remark}[theorem]{Remark}
\newtheorem{definition}[theorem]{Definition}

\numberwithin{equation}{section}

  \input xypic
  \xyoption{all}

 \usepackage{tikz}
\usetikzlibrary{arrows}

 \usepackage{times}

  \usepackage{amsmath, amsthm, amsfonts}

\def \fkz{\mathfrak{z}}

\def \df{\square}
\def \mc{\mathcal}
\def \inv{^{-1}}

\def \cplane{\mathbb{C}}
\def \mff{\mathsf}
\def \indexn{^{(k)}}

\def \half{\frac{1}{2}}

\def \bb {\mathbb}
\def \b{\bar }

\def \p{\partial}

\def \halfplane{\mathsf{h}^\ast}

 \def\CHART{\mathsf{U}}
 
\def\t{\mathfrak{t}}

              \def \bb {\mathbb}
\def \p{\partial}

\def \v {\vskip 0.1in}

\def \df{\square}
\def \mc{\mathcal}
\def \inv{^{-1}}

\def \v{\vskip 0.1in}

\def \cplane{\mathbb{C}}
\def \indexn{^{(k)}}

\def \half{\frac{1}{2}}

\def \bb {\mathbb}
\def \b{\bar }

\def \p{\partial}

\def \halfplane{\mathsf{h}^\ast}
\def \indexm {_{k} }

 \def\CHART{\mathsf{U}}

\def\t{\mathfrak{t}}
\begin{document}

\title[Prescribed Scaler Curvatures for Homogeneous Toric Bundles]
{Prescribed Scaler Curvatures\\ for Homogeneous Toric Bundles}

\author[Chen]{Bohui Chen}
     \address{Department of Mathematics, Sichuan University, Chengdu, 610064, China}
     \email{bohui@cs.wisc.edu}
\author[Han]{Qing Han}
\address{Department of Mathematics,
University of Notre Dame,
Notre Dame, IN 46556, USA}
\email{qhan@nd.edu}
\address{Beijing International Center for Mathematical Research,
Peking University,
Beijing, 100871, China} \email{qhan@math.pku.edu.cn}
\author[Li]{An-Min Li}
\address{Department of Mathematics,
Sichuan University,
 Chengdu, 610064, China}
\email{anminliscu@126.com}
\author[Lian]{Zhao Lian}
\address{Department of Mathematics,
Sichuan University,
Chengdu, 610064, China}
\author[Sheng]{Li Sheng} 
\address{Department of Mathematics,
Sichuan University,
Chengdu, 610064, China}
\email{lshengscu@gmail.com}
\thanks{Chen acknowledges the support of NSFC Grant .\\
${}\quad\ $Han acknowledges the support of NSF
Grant DMS-1404596.\\
${}\quad\ $Li acknowledges the support of NSFC Grant NSFC11521061.\\
${}\quad\ $Sheng acknowledges the support of NSFC Grant NSFC11471225.}

{\abstract
In this paper, we study the generalized Abreu equation on a Delzant ploytope $\Delta \subset \mathbb{R}^2$ and
prove the existence of the constant scalar metrics
of homogeneous toric bundles under the assumption of an appropriate stability.
\endabstract
}

\maketitle


\section{Introduction}\label{sec-Introduction}

It is an important question to study whether there exist K\"ahler metrics of
constant scalar curvatures, or more general, the extremal metrics, in given K\"ahler classes.
The underlying equation is a fourth-order or sixth-order nonlinear complex elliptic equation for
potential functions. The general form of such an equation seems out of reach at the
present time.

In a series of papers \cite{D-1}, \cite{D2}, \cite{D3}, and \cite{D4}, Donaldson studied
this problem on toric manifolds and proved the existence
of metrics of constant scaler curvatures on toric surfaces
under an appropriate stability condition. Later on in \cite{CLS-1} and \cite{CLS-2},
Chen, Li and Sheng proved the existence of metrics of prescribed scaler curvatures
on toric surfaces under the uniform stability condition.

Toric manifolds are a class of K\"ahler manifolds which admit a $(\mathbb C^*)^n$-action
with a dense, fixed-point free orbit and hence admit moment maps with images given by
convex polytopes in $\mathbb R^n$. On toric manifolds, invariant K\"ahler metrics
can be characterized by their symplectic potentials, which are smooth strictly convex
functions on the polytopes. As a result, the fourth-order nonlinear complex elliptic equation
mentioned above is reduced to a fourth-order nonlinear real elliptic equation,
referred to as the Abreu equation.

In \cite{D5}, Donaldson suggested to study the existence K\"ahler metrics
of constant scalar curvatures on toric fibrations. He presented
the underlying equation, which we call the generalized Abreu equation. In this paper we show that the method developed by Chen, Li and Sheng can be used to study   constant scalar curvatures  on homogeneous toric bundles.

The main result in this paper is following theorem, concerning the existence of metrics of prescribed scalar curvatures
on homogeneous toric bundles.

\begin{theorem}\label{theorem_1.1}
Let $(M,\omega)$ be a compact toric surface and $\Delta$ be its
Delzant polytope. Let $G/K$ be a generalized flag manifold with $dim(Z(K))=2$ and
$G\times_KM$ be the homogeneous toric bundle. Let $\mathbb{D}>0$ be the
Duistermaat-Heckman polynomial, and $A\in C^{\infty}(\bar{\Delta})$ be a given smooth function.
Suppose that $\mathbb D$ is an
edge-nonconstant function on $\bar\Delta$ and
$(\Delta,\mathbb{D}, A)$ is uniformly stable. Then, there is a smooth
$(G,  T^{2})$-invariant metric $\mathcal{G}$ on $G\times_{K}M$
such that the scalar curvature of $\mathcal{G}$ is $\mathbb{S}=A+h_G$.
\end{theorem}

Theorem \ref{theorem_1.1} provides an affirmative answer to
the Yau-Tian-Donaldson conjecture for the homogeneous toric bundle
in the case $\mathbb{S}=constant$ and $n=2$.

Refer to Section \ref{sec-Prliminaries} for various notations and terminology, in particular,
the expressions of $\mathbb{D}$ and $h_G$ in \eqref{eqn_D2.7}
and the notions of the uniform stability
in Definition \ref{definition_1.5} and the edge-nonconstant functions in Definition
\ref{definition_1}.

The underlying equation is the following fourth-order PDE:
\begin{equation}\label{eqn-GeneralizedAbreu}
-\frac{1}{\mathbb {D}}\sum_{i,j=1}^n\frac{\partial^2 (\mathbb {D}u^{ij})}{\partial \xi_i\partial \xi_j}
=\mathbb{S}-h_G\quad\text{in }\Delta,
\end{equation}
where $\mathbb{S}$ is the scalar curvature function on $\Delta$ and
$\mathbb {D}$ and $h_G$ are two known functions which are determined by the underling geometry,
with $\mathbb {D}$ strictly positive in $\bar\Delta$. On toric varieties,
we have $\mathbb {D}\equiv1$ and $h_G\equiv0$, and the equation  \eqref{eqn-GeneralizedAbreu}
reduces to the Abreu equation.
In consistence with toric varieties, we write $A=\mathbb{S}-h_G$.

The equation \eqref{eqn-GeneralizedAbreu} was introduced by Donaldson \cite{D5}
in the study of scalar curvatures of toric fibrations. See also \cite{R} and \cite{N-1}.
We call \eqref{eqn-GeneralizedAbreu} the {\it generalized Abreu equation}.
Our aim in this paper is to solve \eqref{eqn-GeneralizedAbreu} in Delzant polytopes for solutions
satisfying the Guillemin boundary condition.

We first note a simple fact. Even if the scalar curvature $\mathbb S$ is constant, the right-hand side
of the equation \eqref{eqn-GeneralizedAbreu} is a function in general, due to the presence of $h_G$.
The difference between the Abreu equation and the generalized Abreu equation seems minor, with
two functions $\mathbb D$ and $h_G$ inserted in the Abreu equation. However, the difference
in the underlying geometry is significant and, as a consequence, the study of the generalized Abreu
equation has become more complicated.

We first review studies of the Abreu equation for dimension 2, with functions in the right-hand side. Donaldson
derived an $L^\infty$-estimate as well as the interior regularity.
Chen, Li and Sheng derived regularity near boundary by a series of techniques
such as affine blowup, differential inequalities on affine invariant quantities and the Bernstein theorem.
They derived regularity near edges and near vertices by different methods. The proof of
the regularity near edges consists of three steps.
First, they derived an estimate of Ricci curvatures near edges
in geodesic balls in terms of the  determinants of the Hessian of solutions. Second, they derived a lower bound of
geodesic distances. Third, they derived estimates of a ratio of two determinants related to the Hessian of solutions.
By combining these three steps, they obtain regularity near edges.
To obtain regularity near vertices, it is crucial to construct subharmonic functions in
the preimages of neighborhoods of vertices by moment maps.

Now we turn our attention to the generalized Abreu equation
\eqref{eqn-GeneralizedAbreu} in $\Delta$.
For $\Delta\subset \mathbb R^2$,
Nyberg \cite{N-1} derived an $L^\infty$-estimate as well as the interior regularity for
toric fibrations.
For general dimension, Li, Lian, and Sheng \cite{LLS} proved that the uniform stability is the
necessary condition for the solvability of the generalized Abreu equation
and implies the interior regularity.
Some estimates and differential inequalities were established in
\cite{LLS-1} and \cite{LSZ}.

Our main concern in this paper is the boundary regularity for the generalized Abreu equation
\eqref{eqn-GeneralizedAbreu}. We will use methods from
\cite{CLS-1} and \cite{CLS-2} to the present setting. Some generalizations from the Abreu equation to
the generalized Abreu equation are straightforward, while others are technical and involve
more complicated calculations, mostly due to the presence of the function
$\mathbb D$ in \eqref{eqn-GeneralizedAbreu}.

The form of \eqref{eqn-GeneralizedAbreu} is hard to analyze. We need to
write it in its equivalent forms. 
The Abreu equation is given by \eqref{eqn-GeneralizedAbreu}  with $\mathbb D=1$ and $h_G=0$.
In this case, we can discuss $u=u(\xi)$ or its Legendre transform $f=f(x)$. The Hessian of $u$ or $f$
plays an important role. In fact,
$[\det(u_{ij})]^{-1}$ satisfies a linearized Monge-Amp\`ere equation
in coordinates $(\xi_1, \cdots, \xi_n)$ and $\log\det(f_{ij})$  also
satisfies a similar linearized Monge-Amp\`ere equation
in coordinates $(x_1, \cdots, x_n)$. We can consider any one of these equations depending on
our purposes. However, we do not have such options
for the generalized Abreu equation \eqref{eqn-GeneralizedAbreu}.
For \eqref{eqn-GeneralizedAbreu},
the Hessian is replaced by $\mathbb D\det(f_{ij})$, which satisfies nice equations
in $(\xi_1, \cdots, \xi_n)$.  The function
$\log[\mathbb D\det(f_{ij})]$ satisfies a linear equation with lower order term in coordinates $(x_1, \cdots, x_n)$.
A similar equation also holds in complex manifolds. Based on these equations, we can
discuss convergence of sequences of $f_k$ in complex manifolds.

The blowup process is only different. Under
appropriate assumptions, we can conclude that
$\frac{\p\log\mathbb D_k}{\p \xi_{i}}$ converges to
zero as $k\to \infty$. Then the limits go back to the Abreu equation. 

Estimates near vertices are different significantly.
For toric varieties, we need to estimate in preimages $U$ of vertices by moment maps.
Now for homogeneous toric bundles, similar estimates need to be established in
$G(U)$, under the group action. It is more complicated to construct subharmonic functions there.

We will follow  \cite{CLS-1} and \cite{CLS-2} closely
in our study of homogeneous toric bundles. Some estimates were
already established in \cite{LLS}, \cite{LLS-1}, and \cite{LSZ}. In this paper,
we will focus only on the difference caused by the presence of $\mathbb D$
in \eqref{eqn-GeneralizedAbreu}.

The paper is organized as follows. In Section \ref{sec-Prliminaries},
we briefly review homogeneous toric bundles and introduce the generalized
Abreu equation. We also review some basic estimates for the generalized
Abreu equation. In Section \ref{sect_5}, we discuss estimates of geodesic distances near
the boundary of polytopes. In Section \ref{sect_EstimateK},
we estimate Ricci curvatures near divisors. In Section \ref{sect_6} and Section \ref{sect_7}, we discuss
an upper bound and a lower bound of a function $H$, respectively.
In Section \ref{sect_Convergence}, we prove
a convergence theorem and we only point out the difference. In Section \ref{sec_8}, we
discuss regularity near boundary and finish the proof of Theorem \ref{theorem_1.1}.
Sections  \ref{sect_5},  \ref{sect_6} and \ref{sect_7} are similar to the corresponding part in \cite{CLS-2}.
We  only present main estimates. For other sections, we provide a little more
detailed discussions and emphasize the difference from the corresponding part in
\cite{CLS-1} and  \cite{CLS-2}.


\section{Preliminaries}\label{sec-Prliminaries}

In this section,
we briefly review homogeneous toric bundles and introduce the generalized
Abreu equation. We also review some basic estimates for the generalized
Abreu equation.

\subsection{Homogeneous Toric Bundles}\label{toric-fibration}

We recall some facts about homogeneous toric bundles and refer to \cite{PS},
\cite{Ar} and \cite{D5} for details.
Let $G$ be a compact semisimple Lie group, $K$ be the centralizer of a torus $S$ in $G$,
and $T$ be a maximal torus in $G$ containing $S$.
Then, $T\subset C(S)=K$ and $G/K$ is a generalized flag manifold.
Denote $\mathfrak{o}=eK$.  Let  $\mathfrak{g}$ (resp.  $\mathfrak{k}$, $\mathfrak{h}$ )
be the Lie algebra of $G$ (resp. $K$, $T$) and $B$ denote the Killing form of $\mathfrak{g}$.
Recall that $-B$ is a positive definite inner product on $\mathfrak{g}$.
There is an orthogonal decomposition with respect to $-B$ given by
$$\mathfrak{g}=\mathfrak{k}\oplus \mathfrak{m},$$  with $Ad(k)\mathfrak{m}\subset \mathfrak{m}$
for any $k\in K.$
The tangent space of $G/K$ at $\mathfrak{o}$ is identified with $\mathfrak{m}$.

Let $Z(K)$ be the center of $K$, which is an $n$-dimensional torus, denoted by $T^{n}$.
Let $(M,\omega)$ be a compact toric K\"{a}hler manifold of complex dimension $n$, where $T^{n}$
acts effectively on $M$. Let $\varrho: K\rightarrow T^{n}$ be a surjective homomorphism.
The homogeneous toric bundle $G\times_{K}M$ is defined to be the space $G\times M$ modulo the relation
$$(gh,x)=(g,\varrho(h)x)\quad\text{for any }g\in G, h\in K, x\in M.$$
Later on, we will omit $\varrho$ to simplify notations.
The space $G\times_{K}M$ is a fiber bundle with fiber $M$ and base space $G/K$,
a generalized flag manifold. There is a natural $G$-action on $G\times_{K}M$   given by
$$g\cdot[h,x]=[gh,x]
\quad\text{for any }g\in G, x\in M,$$
and a natural $T^{n}$-action on $G\times_{K}M$ given by
$$k\cdot[h,x]=[g, k^{-1}x]\quad\text{for any }k\in T^{n}.$$

Denote by $\mathfrak{g}^{\mathbb {C}},\mathfrak{k}^{\mathbb {C}}, \mathfrak{h}^{\mathbb {C}}$
the complexification of   $\mathfrak{g}$, $\mathfrak{k}$, $\mathfrak{h}$, respectively.
Let $R$ be the root system of $\mathfrak{g}^{\mathbb {C}}$ with respect to
$\mathfrak{h}^{\mathbb {C}}$,  we have
$$
\mathfrak{g}^{\mathbb {C}}=\mathfrak{h}^{\mathbb {C}}\oplus \sum_{\alpha\in R}\mathbb C E_{\alpha}.
$$
Then, $R$ admits a decomposition
$$R=R_{K}+R_{M}$$
so that $E_{\alpha}\in \mathfrak{k}^{\mathbb {C}}$ if $\alpha\in R_{K}$ and
$E_{\alpha}\in \mathfrak{m}^{\mathbb {C}}$ if $\alpha\in R_{M}.$
For any $\varphi\in \mathfrak{h}^{*}$, we define $h_{\varphi}\in \mathfrak{h}$ by
$$B(h, h_{\varphi})=\varphi(h)\quad\text{for any }h\in \mathfrak{h},$$
and set
$H_{\varphi}=\sqrt{-1}h_{\varphi}.$
Define
$\mathfrak{t}=\mathfrak{h}\bigcap \mathfrak{z}(\mathfrak{k}^{\mathbb{C}}).$
Consider the restriction map $\kappa:\mathfrak{h}^*\rightarrow \mathfrak{t}^*$ given by
$$\kappa( \alpha )= \alpha|_{\mathfrak{t}}.$$
Set $R_T=\kappa(R)=\kappa(R_{M})$. The elements of $R_T$ are called $T$-roots.

We fix an ordering and let $R^+$ (resp. $R_K^+$)  be the set of positive roots of $R$ (resp. $R_K$).
Set $R^+_M=R^+\setminus R_K^+$. Let $\{\tilde{ H}_{\alpha_1},\cdots, \tilde{H}_{\alpha_n}\}$
be a base of $\mathfrak{t}$. We choose a Weyl basis $e_{\alpha}\in \mathfrak{g}_{\alpha}^{\mathbb{C}}$
of $\mathfrak{g}^{\mathbb{C}}$
such that, for $e_{\alpha}\in \mathfrak{g}_{\alpha}$ and $e_{-\alpha}\in \mathfrak{g}_{-\alpha}$,
\begin{equation}\label{eqn_2.0}
B(e_{\alpha},e_{-\alpha})=1,\;\;\;[e_{\alpha},e_{-\alpha}]=h_{\alpha}.\end{equation}
Set  $$  V_{\alpha}=e_{\alpha}-e_{-\alpha}, \quad{W}_{\alpha}=\sqrt{-1}(e_{\alpha}+e_{-\alpha}).$$
It is easy to see that $  H_{\alpha},  V_{\alpha},  {W}_{\alpha} \in
(\mathfrak{g}_{\alpha}\oplus \mathfrak{g}_{-\alpha}\oplus [\mathfrak{g}_{\alpha},\mathfrak{g}_{-\alpha}])
\cap \mathfrak{g}.$
For any $\alpha\in  {R}^{+}$, we have
$$\alpha= \sum_{j=1}^{n+\ell} M_{\alpha}^{j}\alpha_{j},$$
with $M_{\alpha}^{j}\geq 0.$
Obviously,
\begin{equation}\label{eqn_4.6}
H_{\alpha}= \sum_{j=1}^{n+\ell} M_{\alpha}^{j}H_{\alpha_j}= \sum_{j=1}^{n} M_{\alpha}^{j}\tilde{H}_{\alpha_j}
+ \sum_{j=n+1}^{n+\ell} M_{\alpha}^{'j}H_{\alpha_j}.
\end{equation}
 For any $\alpha\in R_{M^{+}}$, by $M_{\alpha}^{j}\geq 0,$  we have
\begin{equation}
\sum_{j=1}^{n} M_{\alpha}^{j}>0.
\end{equation}

\subsection{The Generalized Abreu Equation}\label{sec-2.2}

For any $1\leq j\leq n$ and $\alpha\in R_M^{+}$,
let $\tilde{H}^*_{\alpha_j} $, $  V^*_{\alpha},  W^*_{\alpha}   $
be the fundamental vector fields  corresponding to  $\tilde{H}_{\alpha_j} $, $V_{\alpha}, W_{\alpha}$.
Then the left-invariant vector fields
$\{\frac{\p}{\p x_{j}}, \tilde{H}^*_{\alpha_j},V^*_{\alpha}, W^*_{\alpha}\}_{1\leq j\leq n,\alpha\in R_M^{+}}$
is a local basis of $G\times_{K}M$.
Let $\{dx_{j},\nu^{j}, dV^{\alpha},dW^{\alpha}\}_{1\leq j\leq n,\alpha\in  R_M^{+}}$
be the dual left-invariant 1-form of the basis.

Denote by $\tau:M\rightarrow \bar{\Delta}\subset \mathfrak{t}^*$
the moment map of $M$, where $\Delta$ is a Delzant polytope.
The left invariant 1-form $\{\nu^{1},\cdots, \nu^{n}\}$ can be seen as a basis of $\mathfrak{t}^{*}.$
The moment map $\tau: M \rightarrow  \mathfrak{t}^{*}$
has components relative to this basis of $\mathfrak{t}^{*}$, which we denote by $\tau_i$.
Note that $\sum_{i=1}^{n} \tau_{i}\nu^{i}$ is independent of the choice of the basis.

We fix a point $o\in \mathbb{R}^n$ and
identify $\mathfrak{t}^*$ with $T_{o}\mathbb{R}^{n}$. We view $o$ as
the origin of $\mathbb R^n$ and $\{\nu^{1},\cdots, \nu^{n}\}$ as
a  basis of $\mathbb{R}^n$.
Let $\xi=(\xi_1,...,\xi_n)$ be the coordinate system with respect to such a basis.
We choose
$$  \bar{\Delta}\subset \{(\xi_1,...,\xi_n)| \xi_1>0,\;\xi_2>0,\;...,\xi_n>0\}$$  such that
\begin{equation}\label{equ_R_3.6}
 \sum_{\alpha\in R_{M^{+}}}\frac{(\sum_{j=1}^{n} M_{\alpha}^{j})diam(\Delta)}{D_{\alpha}}<\frac{n}{4},
\end{equation}
where
$$D_{\alpha}:=2\sum_{i=1}^{n} \tau_{i}\nu^{i}=2\sum_{j=1}^{n} M_{\alpha}^{j}\xi_{j}>0\;\;\forall \;\xi\in \bar{\Delta}.$$
Since the moment map is equivariant, we can also regard $\tau$
as a map from $G\times_{K}M$ to $\mathfrak{t}^*$ and the components $\tau_{i}$ as functions on $G\times_KM$;
that is, we extend $\tau: G\times_{K}M\rightarrow \bar{\Delta}$ by $\tau([g,x])=\tau(x)$.

Suppose that $\Delta$ is defined by linear inequalities $h_k(\xi)-c_k>0$,  for $k=1, \cdots, d$,
where $c_k$ are constants and
$h_k$ are affine functions in $\mathbb  R^n$, $k=1, \cdots, d$, and each $h_k(\xi)-c_k=0$ defines a facet of $\Delta$.
Write $\delta_k(\xi)=h_k(\xi)-c_k$
and set
\begin{equation}\label{eqn 2.5}
v(\xi)=\sum_k\delta_k(\xi)\log\delta_k(\xi).
\end{equation}
It defines a K\"ahler metric on $G\times_{K}M$, which we call the {Guillemin} metric.
For any strictly convex function $u$ with $u-v\in C^{\infty}(\bar{\Delta})$,
we consider the K\"ahler metric $\mathcal{G}_{u}$ on $G\times_{K}M$ defined by
\begin{align*}
\mathcal G_{u}= \sum_{i,j=1}^{n}( u_{ij}d\xi_{i}\otimes d\xi_{j} + u^{ij}\nu^{i}\otimes \nu^{j})
+\sum_{\alpha\in R_{M^{+}}} D_{\alpha}(dV^{\alpha}\otimes dV^{\alpha}+dW^{\alpha}\otimes dW^{\alpha}),
\end{align*}
where $u_{ij}=\frac{\p^2 u}{\p \xi_i \p \xi_j} $ and $(u^{ij})$ is the inverse matrix of $(u_{ij})$. Let $f$ be the Legendre transformation of $u$, i.e.,
$$
x^{i}=\frac{\p u}{\p \xi_{i}},\;\;\;f=\sum_{j=1}^{n} x^{j}\xi_{j}-u.
$$
In term of $x$ and $f$, the metric can be written as
\begin{align*}
\mathcal G_{f}= \sum_{i,j=1}^{n}(f_{ij}dx^{i}\otimes dx^{j} + f_{ij}\nu^{i}\otimes \nu^{j})
+\sum_{\alpha\in R_{M^{+}}} D_{\alpha}(dV^{\alpha}\otimes dV^{\alpha}+dW^{\alpha}\otimes dW^{\alpha}),
\end{align*}
where $f_{ij}=\frac{\p^2 f}{\p x^i \p x^j}.$
One can check that $f$ is the potential function of the metric $\mathcal G_{f}.$

Set
$$S_{j}=\frac{1}{2}(\frac{\p}{\p x_{j}}-\sqrt{-1} H_{j}^*),\;\;\;
S_{\alpha}=\frac{1}{2}(  V^*_{\alpha}-\sqrt{-1}  W^*_{\alpha}),\;\;\;j\leq m, \;\alpha\in R_{M^{+}}.$$
As in \cite{LSZ} (see also \cite{PS}), the Ricci curvatures are given by
\begin{align}\label{eqn_Ric_4.7}
&Ric (S_{j},\bar S_{k})=-\frac{1}{4}\frac{\p^2 log\mathbb F_{\Delta}}{\p x_{j}\p x_{k} },\;\;\;
Ric (S_{\alpha},\bar S_{k})=
Ric  (S_{j},\bar S_{\alpha})=0,\;\;\; \\\label{eqn_Ric_4.8}
& Ric (S_{\alpha},\bar S_{\beta})=
\delta_{\alpha\beta}\left[ -\frac{1}{4}\sum f^{kl}\frac{\p D_{\alpha}}{\p x_{k}}\frac{\p log \mathbb F_{\Delta}}{\p x_{l} }
+ \frac{1}{4} \frac{\p D_{\alpha}}{\p \xi_{k}}\sigma_{k}\right],
\end{align}
where
$\mathbb{F}_{\Delta}=\left(\prod\limits_{\alpha\in R_{M^{+}}} D_{\alpha}\right)\cdot \det(f_{ij}),$
$\sigma $ is the sum of the positive roots of $R^+_M,$ and
$$\sigma_i=-2\sum_{\alpha\in R^{+}_M} \alpha(\sqrt{-1}\tilde {\mathcal H}_{i}).$$
Denote
 \begin{equation}\label{eqn_K1}
\mathbb K\;=\;\|Ric\|_{\mathcal G_{f}} +\|\nabla Ric\|_{\mathcal G_{f}}
^{\frac{2}{3}}+\|\nabla^2 Ric\|_{\mathcal G_{f}}
^{\frac{1}{2}}.
\end{equation}
 The scalar curvature of $\mathcal G_{u}$ is given by
\begin{equation}\label{eqn 4.6}
\mathbb{S}= -\frac{1}{\mathbb {D}}\sum_{i,j=1}^n\frac{\partial^2(\mathbb {D}u^{ij})}{\partial \xi_i\partial \xi_j} + h_G,
\end{equation}
where
\begin{equation}\label{eqn_D2.7}
\mathbb {D}= \prod_{\alpha\in R_{M^{+}}} D_{\alpha},\quad
h_G= \sum_{i=1}^{n} \sigma_{i}\frac{\partial \log \mathbb{D}}{\partial \xi_{i}}.\end{equation}
Here, $\mathbb{D}$ is called the Duistermaat-Heckman polynomial. Set $A=\mathbb{S}-h_G$.

We will consider the equation
\begin{equation}\label{eqn 2.7}
-\frac{1}{\mathbb {D}}\sum_{i,j=1}^n\frac{\partial^2(\mathbb {D}u^{ij})}{\partial \xi_i\partial \xi_j}=A,
\end{equation}
where $\mathbb{D}>0$ and $A$ are given functions defined on $\bar{\Delta}$.
The equation \eqref{eqn 2.7} was introduced by Donaldson \cite{D5}
in the study of scalar curvatures of toric fibrations. See also \cite{R} and \cite{N-1}.
We call \eqref{eqn 2.7} the generalized Abreu Equation.

\subsection{Uniform Stability}\label{sec-2.3}

We introduce several classes of functions. Set
\begin{align*}
\mc C&=\{u\in C(\bar\Delta):\, \text{$u$ is
convex on $\bar\Delta$ and smooth on $\Delta$}\},\\
\mathbf{S}&=\{u\in C(\bar\Delta):\, \text{$u$ is convex  on $\bar\Delta$
and $u-v$ is smooth on $\bar\Delta$}\},\end{align*}
where $v$ is given in \eqref{eqn 2.5}.
For a fixed
point $p_o\in \Delta$, we consider
\begin{align*}
{\mc C}_{p_o}&=\{u\in \mc C:\, u\geq u(p_o)=0\},\\
\mathbf{S}_{p_o}&=\{ u\in \mathbf S :\, u\geq u(p_o)=0\}.\end{align*}
We say functions in ${\mc C}_{p_o}$ and ${\mathbf{S}}_{p_o}$ are {\it normalized} at $p_o$.

Following \cite{N-1}, we consider
\begin{equation}\label{eqn 1.2}
\mc F_A(u)=-\int_\Delta \log\det(u_{ij}) \mathbb{D}d \mu+\mc L_A(u),
\end{equation}
and
\begin{equation}\label{eqn 1.3}
\mc L_A(u)=\int_{\partial\Delta}u \mathbb{D}d\sigma+\int_\Delta Au  \mathbb{D} d\mu,
\end{equation}
where $\mathbb{D}>0$ and $A$ are given smooth functions defined on $\bar{\Delta}$.
$\mc F_A$ is called the Mabuchi functional
and $\mc L_A$ is closely related to the Futaki invariants. The Euler-Lagrangian equation for $\mc F_A$ is
\eqref{eqn 2.7}. It is known that, if $u\in \mathbf{S}$ satisfies the equation \eqref{eqn 2.7},
then $u$ is an absolute minimizer for
$\mc F_A$ on $\mathbf{S}$.

\begin{definition}\label{definition_1.5}
Let $\mathbb{D}>0$ and $A$ be two smooth functions on $\bar\Delta$. Then,
$({\Delta},\mathbb{D}, A)$ is called {\it uniformly $K$-stable} if
the functional $\mc L_A$ vanishes on affine-linear functions and
there exists a constant $\lambda>0$
such that, for any $u\in  {\mc C}_{p_o}$,
\begin{equation}\label{eqn 1.6}
\mc L_A(u)\geq \lambda\int_{\partial \Delta} u \mathbb{D}d \sigma.
\end{equation}
We also say that $\Delta$ is
$(\mathbb{D},A,\lambda)$-stable.
\end{definition}

Li, Lian and Sheng \cite{LLS}  proved that
 \begin{itemize}
\item[(1)] if the equation \eqref{eqn 2.7} has a solution in $\mathbf{S}$, then $(\Delta, \mathbb{D}, A)$ is uniform K-stable,
\item[(2)] the uniform K-stability implies the interior regularity  for any dimension $n$.
\end{itemize}

\begin{remark}  Nyberg \cite{N-1} proved (2) for toric fibration and $n=2.$
\end{remark}

Donaldson \cite{D4} derived an $L^{\infty}$-estimate for the Abreu equation in
$\Delta \subset \mathbb{R}^2$. His method can be applied directly to the
generalized Abreu Equation $\Delta \subset \mathbb{R}^2$. See \cite{N-1}.

\begin{theorem}\label{theorem_3.2}
Let $\Delta\subset \mathbb{R}^2$ be a Delzant polytope and $\mathbb{D}>0$ and
$A$ be two smooth functions defined on $\bar\Delta$. Let $u\in C^{\infty}(\Delta)$ satisfy \eqref{eqn 2.7}.
Suppose that $\Delta$ is $(\mathbb{D},A,\lambda)$-stable. Then,
$$|\max\limits_{\bar \Delta} u-\min\limits_{\bar \Delta} u|\leq \mathsf C_o,$$
where
$\mathsf C_o$ is a positive constant depending on $\lambda$, $\Delta$, $\mathbb{D}$
and $\|A\|_{C^0}$.
\end{theorem}

\begin{definition}\label{definition_1}
Let $K$ be a smooth function on $\bar\Delta$. It is called {\it edge-nonconstant}
if it is not constant  on any edge of $\Delta$.  It is called {\em edge-nonvanishing}
if it does not vanish on any edge of $\Delta$.
\end{definition}

\subsection{Some Estimates for the Generalized Abreu Equation}\label{sec-2.4}

We adopt notations in \cite{CLS-1}, \cite{CLS-2}.
Let $(M,\omega_o)$ be a compact toric K\"{a}hler manifold of complex dimension $n$
and denote by $\tau:M\rightarrow \bar{\Delta}\subset \mathfrak{t}$ the moment map of $M$,
where $\Delta$ is a Delzant polytope.
Suppose that $\omega_o$ is the Guillemin metric with the local potential function
${\bf g}$. For any $T^n$-invariant metric $\omega\in [\omega_o]$ with the local potential
function ${\bf f}$, there is a function $\phi$ globally defined on $M$ such that
$${\bf f}={\bf g} + \phi.$$
Set
$$
\mc C^\infty(M,\omega_o)=\{\mathsf f|\mathsf f=\mathsf g+\phi,  \phi\in C^\infty_{\bb T^2}(M) \mbox{ and } \omega_{\mathsf f}>0\}.
$$
Fix a large constant $K_o>0$. We set
$$
\mc C^\infty(M,\omega_o;K_o)
=\{\mathsf{f}\in \mc C^\infty(M,\omega_o)|
|\mathcal{S}(f)|\leq K_o\}.
$$
Take a point $p_o\in \bar{\Delta}$. We choose coordinates $\xi_1,...,\xi_n$ such that $\xi(p_o)=0$.
Let $u=u(\xi_1,...,\xi_n)\in \mathbf{S}$ be a solution of \eqref{eqn 2.7}.
Set
$$x_i=\frac{\partial u}{\partial \xi_i},\;\;\; f=\sum_i x_i\xi_i - u,$$
and
$$w_i= x_i + \sqrt{-1}y_i,\;\;z_i=e^{w_i/2}.$$
In the case $p_o\in \p \Delta$, it is easy to show that the potential function $f$ can be
extended smoothly to the  divisors (see \cite{CLS-1} and \cite{CLS-2}). The Ricci curvature
and the scalar curvature of the K\"{a}hler metric
$\omega_f^{M}$ are given by
$$
R^{M}_{i\b j}= - \frac{\partial^{2}}{\partial z_{i}\partial \bar
z_{j}} \left(\log \det\left(f_{k\bar l}\right)\right), \;\;\;\mc
S^{M}=\sum f^{i\bar j}R_{i\b j},
$$
respectively. When we use the log-affine coordinates
on $\t$,   the Ricci curvature
and the scalar curvature of $\omega_f^{M}$ can be written as
$$
R^{M}_{i\b j}= - \frac{\partial^{2}}{\partial x_{i}\partial x_{j}}
\left(\log \det\left(f_{kl}\right)\right), \;\;\;\mc S^{M}=-\sum
f^{ij}\frac{\partial^{2}}{\partial x_{i}\partial x_{j}} \left(\log
\det\left(f_{kl}\right)\right).
$$
Define
\begin{equation}\label{eqn_1.1c}
\mc K\;=\;\|Ric^{M}\|_f +\|\nabla Ric^{M}\|_f
^{\frac{2}{3}}+\|\nabla^2 Ric^{M}\|_f
^{\frac{1}{2}}.
\end{equation}
Set
$$
\mathbb{F}_{\Delta}= {\mathbb D} [\det(u_{ij})]^{-1}= \mathbb D\det(f_{ij}).$$
We can rewrite \eqref{eqn 2.7} in the coordinates  $(x^{1},\cdots, x^{n})$ as
\begin{equation}\label{eqn 2.13}
-\sum_{i,j} f^{ij}(\log \mathbb F_{\Delta})_{ij} -\sum_{i,j}f^{ij} (\log \mathbb D)_{i}(\log \mathbb F_{\Delta})_{j} = A.
\end{equation}
On toric varieties, $\mathbb D\equiv0$ and then $\det(f_{ij})$ satisfies a linearized
Monge-Amp\`ere equation. For the present case, first-order terms appear in \eqref{eqn 2.13}.

Lemma 4.2 and Lemma 4.4 in \cite{CLS-2} and Proposition 3.9 and Proposition 4.2 in \cite{CLS-1}
can be extended to the
generalized Abreu equation. The following lemma is proved  in \cite{LLS}.

\begin{lemma}\label{lemma 2.5} Let $\Delta \subset \mathbb{R}^{n}$ be a Delzant polytope,
$\mathbb{D}>0$ and $A$
be smooth functions on $\bar\Delta$. Suppose that $u\in C^{\infty}(\Delta)$ satisfies
the generalized Abreu equation
\eqref{eqn 2.7} and that $\mathbb F=0$ on $\partial \Delta$. Let $E$ be a facet of
$\Delta$ given by $\xi_1=0$ and let $p\in E^o$.
Then, in a neighborhood of $p$,
$$ det(D^2u)\geq \frac{b}{\xi_1} $$
where $b$ is a positive constant depending only on $n$, $diam(\Delta)$,
$\max_{\bar\Delta}\mathbb D$, $\min_{\bar\Delta}\mathbb D$ and $\|A\|_{L^{\infty}(\Delta)}$.
\end{lemma}

The following results are proved in \cite{LLS-1}.

\begin{lemma}\label{lemma_2.8} Let $\Delta\subset \mathbb R^2$ be a Delzant ploytope, $\mathbb{D}>0$ and $A$
be smooth functions on $\bar\Delta$. Suppose that $u\in \mathbf{S}_{p_o}$ satisfies the generalized Abreu equation
\eqref{eqn 2.7}, and suppose that there are two constants
$b,d>0$ such that
\begin{equation}\label{eqn_2.0.a}
\frac{\sum \left(\frac{\partial u}{\partial \xi_k}\right)^2}{(d+
f)^2}\leq b,\;\;\;\; d+f\geq 1
\end{equation} where  $f$ is the Legendre function of $u.$
Then,
$$\frac{\det(\partial^2_{ij}u)}{(d+f)^4}(p) \leq \frac{b_0}{ d_E(p, \partial
\Delta)^{4}},$$
then $b_0$ is a positive constant depending $diam(\Delta)$, $\max_{\bar\Delta}\mathbb D$,
$\min_{\bar\Delta}\mathbb D$ and $\|A\|_{L^{\infty}(\Delta)}$.
\end{lemma}

\begin{theorem}\label{theorem_5.1}
Let $\Omega^*\subset \mathbb{R}^2$ be a  normalized domain. Let $u\indexm\in \mc
F(\Omega^*,1)$ be a sequence of functions and $p^o\indexm$ be
the minimal point of $u\indexm$. Let $\mathbb D_{k}>0$ be
a given smooth function defined on $\overline{\Omega}^*.$
Suppose that there is $\mff N_{1}>0$ such that
$$|\mathcal{S}_{\mathbb{D}_{k}}(u\indexm)|\leq \mff N_{1},\;\;\;\;\mff N_{1}^{-1} \leq \mathbb D_{k}\leq \mff N_{1}$$
and
$$
 \sup_{\Omega^*}|\nabla_{\xi}\log \mathbb D_{k}|\leq \mff N_{1}.
$$
Then, up to subsequences,
$u\indexm$ locally uniformly  converges to a function $u_\infty$ { in $\Omega^*$} and
$p_{k}^o$ converges to $p^o_\infty$ such that $$d_E(p^o_{\infty},\partial\Omega^*)>\mff s$$
for some constant $\mff s>0$ and in
$D_\mff s(p_{\infty}^o)$
$$
\|u\|_{C^{3,\alpha}}\leq\mff  C_1
$$
for some  $\mff C_1>0$ and $\alpha\in (0,1)$.
\end{theorem}

Let $u(\xi)$ be a smooth, strictly convex function defined on a
convex domain  $\Omega^{*}\subset  \t^{*}$.  As $u$ is strictly
convex,
$$
 G_u=\sum_{i,j} u_{ij}d \xi_i d\xi_j
$$
defines a Riemannian metric on $\Omega$. We call it the {\em Calabi} metric. \v
For any $p\in \bar \Delta,$ denote
 $B_b(p, \Delta)=\{ q\in  {\Delta}|\; d_{u}(q,p)< b\}, $  where $d_u(p,q)$ is the distance from $p$ to $q$
with respect to the Calabi metric $G_u$. Sometimes we call it the intersection of the geodesic ball
$B_{b}(p)$ and $\Delta$, and  denote by $B_{b}(p)\cap \Delta.$

\begin{theorem}\label{theorem_4.2.1}
Let $u\in \mathbf S$. Choose a coordinate system $(\xi_1,\xi_2)$ such that $\ell=\{\xi|\xi_1=0\}.$
Let $p\in \ell^\circ$  such that $B_b(p, \Delta)$ intersection
with $\partial\Delta$ lies in the
interior of $\ell$. Suppose that
\begin{equation}\label{eqn_7.1}
\| {\mathbb{S}}(u)\|_{C^3(B_b(p, \Delta))}\leq \mff N_{2} ,  \quad
h_{22}|_{\ell}\geq \mff N_{2} \inv,
\end{equation}
for some constant $\mff N_{2}>0,$ where $h=u|_{\ell}$ and $\|.\|_{C^3(\Delta)}$
denotes the Euclidean  $C^{3}$-norm.
Then, for any $p\in B_{b/2}(p, \Delta) $,
\begin{equation}\label{eqn_4.4}
\left(\Theta+\mathcal K+\mathbb K \right)(p)
d^2_u(p,\ell )\leq \mff C_2,
\end{equation}
where $\mff C_2$ is a positive constant
depending only on   $\mff N$. \end{theorem}

\section{Estimates of Riemannian Distances on $\partial\Delta$}
\label{sect_5}

In this section, we discuss estimates of geodesic distances near
the boundary of polytopes.
We proceed similarly as in Section 5 \cite{CLS-2} and we only state results.

Let $\Delta\subset \mathbb R^2$ be a Delzant polytope, $\ell$ be an edge of $\Delta$ and
$\ell^\circ$ be the interior of $\ell$. Let $\halfplane\subset \t^\ast$ be the half plane given by
$\xi_1\geq 0$. Take a point $\xi^{(\ell)}\in \ell^{\circ}$.
For simplicity, we fix a coordinate system on
$\t^\ast$ such that { (i) $\ell$ is on the $\xi_2$-axis; (ii) $\xi^{(\ell)}=0$;
(iii) $\Delta\subset \halfplane$.} Define
$$\ell_{c,d}=\{(0,\xi_2)|c\leq \xi_2\leq d\}\subset \ell^o.$$

Let $u_{k} \in \mc C^\infty(\Delta,v;K_o)$ be a sequence of
functions satisfying
\begin{equation}\label{eqn_3.1}
-\frac{1}{\mathbb {D}}\sum_{i,j=1}^n\frac{\partial^2 \mathbb {D}(u_{k})^{ij}}{\partial \xi_i\partial \xi_j}=A_{k}.
\end{equation}
We define the operator $\mathbb S_{\mathbb{D}}(u)$   as
$$
\mathbb S_{\mathbb{D}}(u)=-\frac{1}{\mathbb {D}}\sum_{i,j=1}^n
\frac{\partial^2 \mathbb {D} u^{ij}}{\partial \xi_i\partial \xi_j}.
$$
 Suppose that $\Delta$ is uniformly $(\mathbb{D}, A_{k}, \lambda)$-stable and that $A_{k}$
$C^3$-converges to $A$ on $\bar \Delta$. Then,
by the interior regularity  and Theorem  \ref{theorem_3.2}, we have
\begin{enumerate}
\item $\left|\max_{\bar\Delta} u_k-\min_{\bar\Delta} u_k\right| \leq \mathsf C_o$,
\item  $u_{k}$  locally   $C^6$-converges
in $\Delta$ to a strictly convex function $u_\infty$ and $u_\infty$ can be
continuously  extended to be defined on $\bar \Delta.$
\end{enumerate}
Denote by $h_{k}$ the restriction of $u_{k}$ to $\ell$. Then, $h_{k}$ locally
uniformly converges to a convex function $h $ on $\ell$.
Obviously, $u_\infty|_{\ell^\circ} \leq h $. By the estimates of determinantes provided
by  Lemma \ref{lemma 2.5} and
Lemma \ref{lemma_2.8} and the same argument as in Subsections \S 5.1-\S 5.2\cite{CLS-2},
we get the $C^0$-convergence and the strict convexity of $h$ as follows.

\begin{lemma}\label{proposition_5.1.1}
For $q\in \ell^\circ$,
$u_\infty(q) = h(q)$.
\end{lemma}

\begin{lemma}\label{lemma_5.2.1}
Let $u\in \mathbf S$ and $h=u|_{\ell}$. There is
a constant
 $\mff C_3>0$ depending only   $diam(\Delta)$, $\max_{\bar\Delta}\mathbb D$,
 $\min_{\bar\Delta}\mathbb D$ and $\|A\|_{L^{\infty}(\Delta)}$, such that on
$\ell_{c,d}$, $ \partial^2_{22}h\geq \mff C_3.$\end{lemma}

By Theorem \ref{theorem_5.1}, Theorem \ref{theorem_4.2.1} and by using the same method as in
Subsections \S 5.3-\S 5.4\cite{CLS-2} (see \cite{LLS-1} for more explanation),  we can prove the following result.

\begin{theorem}\label{theorem_5.3.3}
Let $p\in \ell_{c,d}.$
Suppose that  $D_a(p)\cap \partial \Delta\subset  \ell_{c-\epsilon_o,d+\epsilon_o}$ for some $0<a<\epsilon_{o}.$
Then, there exists a small constant
$\epsilon$, independent of $k$, such that the intersection of the geodesic ball
$B\indexn_\epsilon(p)$ and $\Delta$ is contained in a  Euclidean   half-disk $D_a(p)\cap \Delta$.
Here  $D_a(p)$  and $B\indexn_a(p)$ denote the balls of radius $a$ that are  centered at
$p$ with respect to the Euclidean metric and Calabi metric $G_{u_{k}}$ respectively.
\end{theorem}

\section{Estimates of $\mc K$ near Divisors}\label{sect_EstimateK}

In this section, we estimate Ricci curvatures near divisors.
We follow the corresponding section in \cite{CLS-1}.

Consider a polytope $\Delta\subset \mathbb R^2.$
Let $\fkz_o$ be a point on a
divisor $Z_\ell$ for some $\ell$ and consider $\mff{p}_{o}\in \ell$ with $\mff{p}_{o}=\tau(\fkz_{o})$. We will denote $Z:=Z_\ell$. Take a coordinate transformation
$$
\xi^{\ell}_{i}=\sum_{j=1}^{2} a_{i}^{j}(\xi_{j} -\xi_{j}(\mff{p}_{o}))
$$
so that $\xi^{\ell}(\mff{p}_{o})=0$ and $\ell=\{\xi^{\ell}|\xi^{\ell}_{1}=0\}.$
Let $f_{\ell} $ be the Legendre transformation of $u$ in terms of $\xi^\ell$, i.e.,
$$
x^{i}_{\ell}=\frac{\p u}{\p \xi_{i}^{\ell}},\;\;\; f_{\ell}=\sum _{i=1}^{2} \xi^{\ell}_{i}x_{\ell}^{i}-u.
$$
Then, $f_{\ell}(\fkz_{o})=\inf f_{\ell}=0$ and
$$f_{\ell}=f-\sum a_{i}x^{i}-c_{o},$$
where $f$ is as in section \S\ref{sec-2.2}. Choose a local complex coordinate system
$(z_{\ell}^1,z_{\ell}^2)$ around $\fkz_o$ such that
$$z_{\ell}^1=e^{\frac{1}{2}(x_{\ell}^{1}+\sqrt{-1}y_{\ell}^1)} ,\;\;\;z_{\ell}^2=x_{\ell}^{2}+\sqrt{-1}y_{\ell}^2.$$
Set
\begin{equation} \label{eqn_f_5.2}
 W_{\ell}=64\det\left(\frac{\p ^2 f_{\ell}}{\p z^i_{\ell} \p \bar  z^j_{\ell}}\right)
 =\det\left(\frac{\p ^2 f_{\ell}}{\p x^i_{\ell} \p x^j_{\ell}}\right)e^{-x^1_{\ell}} ,\;\;\;\;\;\mathbb F_{\ell}=W_{\ell}\mathbb D,
\end{equation}
and
\begin{equation}
 \label{eqn_f_5.3}
 \mathbf{\Psi}_{\ell}:=\|\nabla log \mathbb F_{\ell}\|^2_f,\;\;\;\;\;\;\;\;\;\;\;\;\;
P=\exp\left(\kappa
\mathbb  F_{\ell}^{a}\right)\sqrt{\mathbb  F_{\ell}}{\mathbf{\Psi}_{\ell}},
\end{equation}
where $\kappa$ and $a$ are constants to be determined later. Then by \eqref{eqn 2.13}, we have
\begin{align}\label{eqn_4.3}
-\square \log \mathbb  F_{\ell} =A+ \sum_{k=1}^{2}\frac{\p \log \mathbb  D}{\p \xi _{k}}
\frac{\p x^{1}_{\ell}}{\p x^{k}} =A+ \sum_{k=1}^{2}\frac{\p \log \mathbb  D}{\p \xi_{k}}  b^{1}_{k}  :=  \mathbb  A_{\ell},
\end{align}
where $(b_{i}^{j})$ denote the inverse matrix of $(a_{i}^{j})$.
We point out that the equation \eqref{eqn_4.3} involves first-order terms
due to the presence of $\mathbb D$.

In this section, we denote $\mathbb F_{\ell }, \mathbf{\Psi}_{\ell},\mathbb A_{\ell},\cdots$
by $\mathbb F,\mathbf{\Psi},\mathbb A,\cdots $ to simply notation.
We proceed similarly as in \cite{CLS-1} to prove the following result.

\begin{theorem}\label{theorem_5.0.1}
Let $u\in \mathbf S$. Choose a coordinate system $(\xi_1,\xi_2)$ such that $\ell=\{\xi|\xi_1=0\}.$
Let $p\in \ell^{o}$ and $D_b(p)\cap \bar \Delta$ be an Euclidean half-disk such that its intersection
with $\partial\Delta$ lies in the
interior of $\ell$. Let $B_a(\fkz_o)$ be a closed geodesic
ball satisfying  $\tau_f(B_a(\fkz_o))\subset D_b(p)$. Suppose that
\begin{align}\label{eqn_7.1}
\min_{D_b(p)\cap \bar \Delta}|\mathbb  A(u)|\geq \delta>0,\quad
\|\mathbb {A}(u)\|_{C^3(D_b(p)\cap \bar \Delta)}\leq\mff N_3,  \quad
h_{22}|_{D_b(p)\cap\ell}\geq \mff N_3\inv,
\end{align}
for some constant $N_3>0,$ where $h=u|_{\ell}$ and $\|.\|_{C^3(\Delta)}$
denotes the Euclidean  $C^{3}$-norm.
Then,
\begin{equation}\label{eqn_7.2}
 \frac{\mathbb F^{\frac{1}{2}}(\fkz)}{\max\limits_
 {B_a(\fkz_o )} {\mathbb F}^{\frac{1}{2}}}
\left(\mathbb K(\fkz)+\|\nabla \log |\mathbb A|\|^2(\fkz)+ \mathbf{\Psi}(\fkz)\right)a^2
\leq \mff C_{4}\quad\text{for any } \fkz\in B_{a/2}(\fkz_o ),
\end{equation}
where $\mff C_{4}>0$ is a positive constant depending only on $a$, $b$, $\delta$, and $\mff N_3$.
\end{theorem}

\def \fkz{\mathfrak{z}}

\subsection{Uniform Control of Sections}\label{sect_7.3a}

In order to use affine blow-up technique we need to control sections (see Subsection \S7.2\cite{CLS-1}).  For the generalized Abreu equation we need more arguments due to the presence of the function
$\mathbb D$.

We consider functions
$
u\in \mc C^\infty(\halfplane,v_{\halfplane} )
$; namely $u=\xi_1\log\xi_1 + \xi_2^2+ \psi$ is strictly convex
for some $\psi\in C^\infty(\halfplane)$.
We assume, in addition,
\begin{equation}\label{eqn_7.6}
\Theta_u(p) d_u^2(p, \p \halfplane)\leq \mff C_2,
\end{equation}
where $d_{u}(p,\p \halfplane)$ is the   distance  from $p$ to $\p \halfplane$
with respect to the Calabi metric $G_{u}.$

Let $\fkz^\circ \in
\CHART_{\halfplane}$ be any point such that $d(\fkz^\circ,Z)=1$ and
$\fkz^\ast\in B_1(\fkz^\circ)\cap Z$, where $d (\fkz^\circ,Z)$ is the   distance
from $\fkz^\circ $ to $Z$ with respect to the   metric $\mathcal G_{f}.$
Without loss of generality, we assume that $\fkz^\circ$ is a representative point
of its orbit and assume that it is on $\t$.
Let $p^\circ$ and $p^\ast$ be
their images of moment map $\tau_f$, respectively.

By adding a linear function we normalize $u$ such that $p^\circ$ is the minimal point of $u$, i.e.,
 \begin{equation}\label{eqnc_7.4}
  u(p^\circ)=\inf u.
 \end{equation}
Let $\check p$ be the minimal point of $u$ on $\p \halfplane$ which is the boundary of $\halfplane$.
By adding some constant to $u$, we may require that
\begin{equation}\label{eqnc_7.5}
u(\check p)=0,
\end{equation}
and, by a coordinate translation, we may assume that \begin{equation}\label{eqnc_7.6} \xi(\check p)=0.
\end{equation}
See Figure 1.
 \begin{figure}
\includegraphics[height=2in]{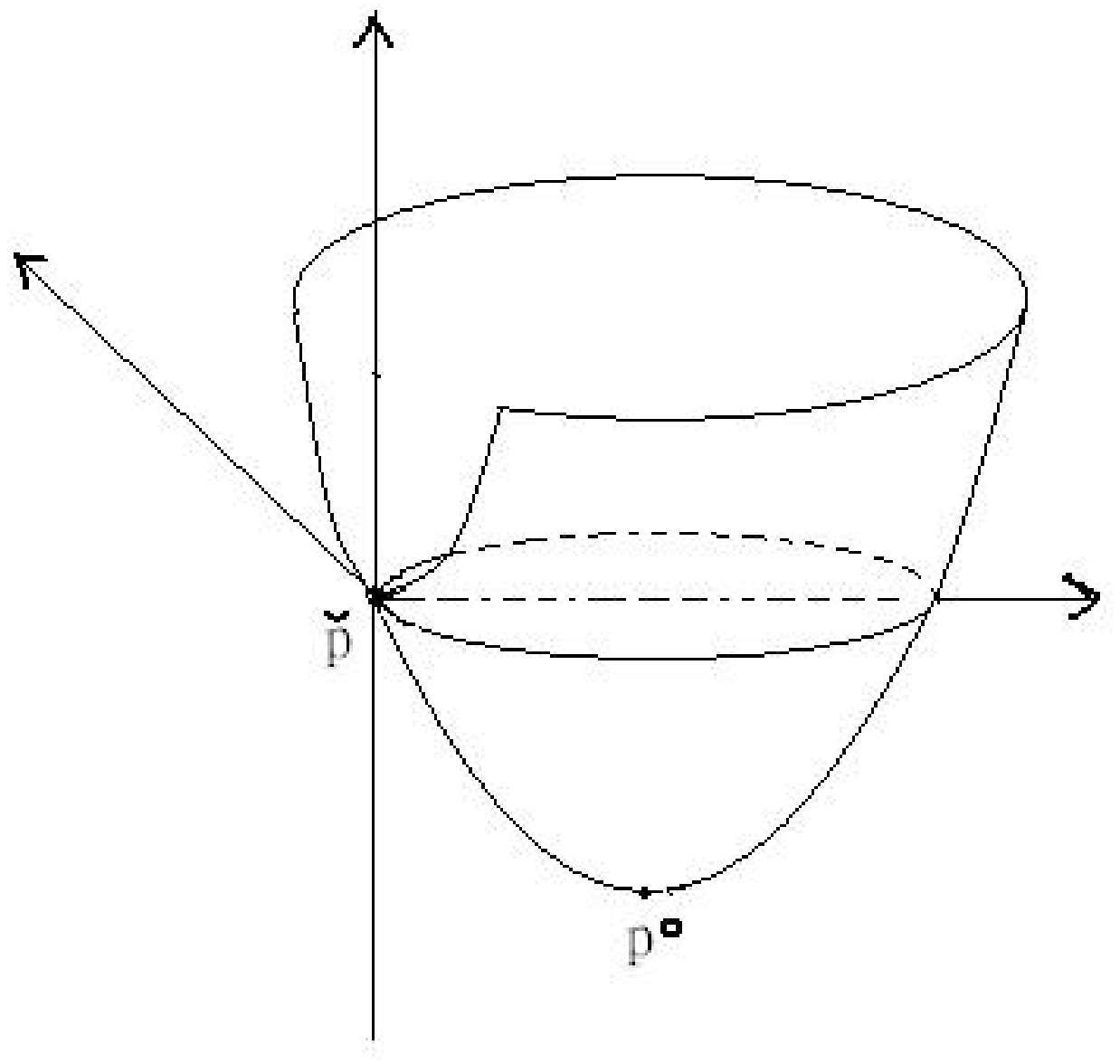}
 \caption{\label{fig1}}
 \end{figure}
Set $z^\ast=\nabla ^u (p^\ast)$ and
$$S_0:=\left\{(-\infty,x_2)\in \t_2 \;|\; \left|\int^{x_2(z^\ast)}_{x_2} \sqrt{f_{22}}dx_2\right|\leq 1 \right\}.$$
By a coordinate transformation
\begin{equation}\label{eqnc_7.3}A(\xi_1,\xi_2)=(\xi_1,\beta\xi_2),\end{equation}
we can normalize $u$ such that
\begin{equation}\label{eqn_7.7}
\left|S_0\right|= 10.
  \end{equation}

We say
$(u,p^\circ, \check p)$ is a {\em minimal-normalized triple}
if $u$ satisfies \eqref{eqnc_7.4}, \eqref{eqnc_7.5}, \eqref{eqnc_7.6},  \eqref{eqn_7.7}, and
$d(p^\circ, \p \halfplane)=1.$

\begin{definition}\label{definition_7.3.3}
Let $(u,p^\circ,\check p)$ be a minimal-normalized triple. Let
$N $ be a constant sufficiently large.
We say that $(u,p^\circ, \check p)$ is a {\it bounded-normalized triple}
if
 \begin{equation}
 \mc K(z)\leq 4, \quad\text{for any } z\in \tau_f\inv(B_N(p^\circ)),\label{eqn_7.8}
 \end{equation}
and
 \begin{equation}
 \frac{1}{4}\leq \frac{W(z)}{W(z')} \leq 4\quad\text{for any } z,z'\in \tau_f\inv(B_N(p^\circ)).\label{eqn_7.9}
 \end{equation}
\end{definition}

Then by the same argument in Subsections \S 7.2-\S 7.3 \cite{CLS-1},  we can prove the following result.
In fact, only Lemma 7.6 in \cite{CLS-1} needs more explanation, which we gave in  \cite{LLS-1}.

\begin{theorem}\label{theorem_7.4.1} Let $(u\indexm, p^\circ\indexm,
\check p\indexm)$ be a sequence of
bounded-normalized triples. Suppose that $\lim_{k\to
0} |\mathbb S(u\indexm)|\to 0$.
 Then,
 \begin{enumerate}
 \item[(1)]  $p^\circ\indexm$ converges to a point $p^\circ_\infty$ and $u_k$
$C^{3,\alpha}$-converges to a strictly convex function $u_\infty$
in $D_{s}(p^\circ_{\infty})
\subset\halfplane$, where $s$ is a constant independent of $k$;
\item[(2)]
 $\fkz^\ast\indexm$ converges to a point $\fkz^\ast_\infty$ and $f \indexm$
   $C^{3,\alpha}$-converges to a
 function $f_\infty $ in $D_{a_1}(\fkz_\infty^\ast)$,
where   $a_1>0 $ is a constant  independent of $k.$
 \end{enumerate}
\end{theorem}

\def \half{\frac{1}{2}}

\subsection{Proof of Theorem \ref{theorem_5.0.1}}\label{sect_7.3b}

\begin{proof}[Proof of Theorem \ref{theorem_5.0.1}]
Set
$$
\mc W_f=\frac{\mathbb F}{\max_{B_a(\fkz_o)}\mathbb F},\;\;\;
\mc R_f= \mc K(f)+\mathbb K(f) +\|\nabla\log |\mathbb A(f)|\|^2+ \mathbf{\Psi}(f).
$$
We note that $\mathbb K(f)$ is extra compared with the corresponding
expression $\mc R_f$ on toric varieties as in \cite{CLS-1}.

Suppose that the theorem is not true. Then, there is a sequence of functions
$f\indexm$  and a sequence of points  $\fkz_k'\in B_{a/2}(\fkz_o)$  such that
\begin{equation}\label{eqn_7.18}
   \mc W_k^{\half} \mc
  R_k(\fkz_k')a^2   \to \infty\;\;\; \mbox{ as } k\to \infty,
\end{equation}
where $\mc W_k=\mc W_{f\indexm}$ and $ \mc R_k=\mc R_{f\indexm}$.
Note that $\mc W_k\leq 1$  in $B_a(\fkz_o)$.
Consider the function
$$F\indexm(z )=  \mc W_k^{\half} \mc R_k(z)
 [ d_{f\indexm}(z,\partial
B_a(\fkz_o))]
 ^2$$
defined in $B_a(
 \fkz_o),$ where $B_a(
 \fkz_o)$ and $d_{f\indexm}(z,\partial
B_a(\fkz_o))$ denotes the geodesic ball and the geodesic distance with respect to the
metric $\mathcal G_{f\indexm}$. Suppose that it
attains its maximum at $z^o\indexm$. By \eqref{eqn_7.18}, we have
\begin{equation}\label{eqn_7.20}
\lim_{k\to\infty}F\indexm(z^o\indexm)\to +\infty, \;\;\;\;
\lim_{k\to\infty}\mc W_k^{\half} \mc R_k(z^o\indexm)=+\infty, \;\;\;\lim_{k\to\infty} \mc R_k(z^o\indexm)=+\infty.
\end{equation}
Set $$d\indexm=\frac{1}{2} d_{f\indexm}(z^o\indexm,\partial B_a
(\fkz_o)).$$ Then, 
$$
{\mathbb F_k}^{\frac{1}{2}}\mc R_k \leq 4 {\mathbb F_k}^{\frac{1}{2}}\mc R_k(z_k^o)
\quad\text{in } B_{d_k}(z_k^o).
$$
Let $z_k^\ast\in Z$ be the point such that $ d_{f_{k}}(z_k^o,
z_k^\ast)=d_{f_{k}}(z_k^o,Z ).$ Denote by $q_k^o,q_k^\ast, \ldots,$ the  images
of $ z_k^o, z_k^\ast, \ldots,$ under the moment map $\tau_f$.

Now we perform the affine
blow-up analysis to derive a contradiction.
We consider the following
affine transformation on $u$:
\begin{equation}\label{eqn_4.16}
\check  u(\xi)=\lambda u(A\inv(\xi))+\eta\xi_1,
\end{equation}
where $ A(\xi_1,\xi_2)=(\lambda\xi_1,\beta \xi_2  ). $ Let $\check
f=L(\check u)$.
Set
\begin{equation*} 
\lambda_k= \beta_k^2=\mc R_k(z_k^o),\;\;\;
\eta_k= (1+l)\log \alpha_k+\log \mathbb F(z_k^o).
\end{equation*}
where $l$ is the number of $R_{M^{+}}$. Corresponding to the $u_k\to \check{u}_k$, we also have changes $f_k\to \check{f}_k$, $d_k\to \check{d}_k$, etc.
Then,
$$\check x^{1}=x^{1},\quad\check x^{2}=\frac{\lambda}{\beta}x^{2}.$$ This transformation induces
$$
d\check x^{1}=dx^{1},\quad d\check x^{2}=\frac{\lambda}{\beta}dx^{2},\quad
\check  \nu^{1}= \nu^{1},\quad\check \nu^{2}=\frac{\lambda}{\beta} \nu^{2}.
$$
Let $\{\check H_{i}\in \mathfrak{t},i=1,2\}$ be the dual of $\{\check \nu^{i}\}$.Then,
$$\check H_{1}=\tilde H_{1},\quad\check H_{2}=\frac{\beta}{\lambda}\tilde H_{2}.$$
Let $\check M_{\alpha}^{j}$ be the constants such that
$H_{\alpha}|_{\mathfrak{t}}=\sum_{j=1}^{2}\check M_{\alpha}^{j}\check H_{j}$. Then,
$$
 \check D_{\alpha}=\sum \check M_{\alpha}^{j}  \frac{\p \check f}{\p \check x^{j}}
 =\lambda\sum   M_{\alpha}^{j}  \frac{\p f}{\p x^{j}}=\lambda D_{\alpha}.
$$
It induces the metric transformation
\begin{align*}
\mathcal G_{\check f}&= \sum_{i,j=1}^{2}(\check u_{ij}d\check \xi_{i}\otimes d\check \xi^{j}
+\check u^{ij}\check \nu^{i}\otimes \check\nu^{j}) +\sum_{\alpha\in R_{M^{+}}}
\check D_{\alpha}(dV^{\alpha}\otimes dV^{\alpha}+dW^{\alpha}\otimes dW^{\alpha})\\
&=\lambda \mathcal G_{u},
 \end{align*}
where $\check u_{ij}=\frac{\p^2 \check u}{\p \check \xi_{i}\p \check \xi_{j}} $
and $(\check u^{ij})$ denotes the inverse of $(\check u_{ij}).$ Obviously,
$\check{\mathbb D}=\lambda^{l}\mathbb D$.

  A direct calculation yields
$$
 \check {\mathbb F}(z)=\lambda^{l+1} e^{-\eta}  \mathbb F( B\inv_{\cplane}z),\;\;\;\;
 \check  {\mathbf{\Psi}}(z)=\lambda\inv \mathbf{\Psi}(B\inv_{\cplane}z), \;\;\;\;
 \check  {\mathbb{K}}(z)=\lambda\inv   {\mathbb{K}}(B\inv_{\cplane}z),$$
and
 $$
\|\nabla\log |\check {\mathbb A}|\|^2_{\check f}(z)=\lambda\inv \|\nabla\log |{\mathbb A}|\|^2_{f}(B\inv_{\cplane}z),
\;\;\;\;\check {\mathcal R}(z)=\lambda^{-1} {\mathcal R}(B\inv_{\cplane}z).$$

We claim that
\begin{enumerate}
\item[(1)] $\lambda_k\to \infty$, $ \check  d_k\to \infty$ as $k\to \infty$;
\item[(2)]$\lim\limits_{k\to\infty}\max\limits_{B_{\check d_k}(\check z^o_k)}
|\check { \mathbb A}_k| = 0$;
\item[(3)] $\mathbb{  \check{\mathbb F}}(\check z^o_k)=1$;
{\item[(4)] $\mathbb{ \check{\mathbb F}}_k^\half\check {\mc R}_k(\check z_k^o)=1$ and
$   \check{\mathbb F}_k^\half\check {\mc R_k}\leq 4 \check{\mathbb F}_k^\half\check{\mc R}_k(z_k^o)=4$ in
$B_{\check d_k}(\check z^o_k)$;
\item[(5)] $ \check{\mathbb F}_k^{\half}\check {\mathbf{\Psi}}_k\to 0$ in
$B_{\check d_k/2}(\check z^o_k)$};
\item[(6)] $\check z_k^\ast \in B_{\sqrt{\mff C_2}}(\check z_k^o).$
 \end{enumerate}

 Note that the proof of (1)-(4) and (6) are the same as in \cite{CLS-1}. To prove (5),
 we need the following lemma, the proof of which can be found in \cite{LSZ}.

\begin{lemma}\label{lemma_7.16a} Let $o\in \tau^{-1}(\ell^{o})$ and  $B_{a}(o) $ be a closed
geodesic ball of radius $a$ centered at o with respect to $\mathcal G_{f}$. Let $\mathbb F_\diamond:=\max\limits_{B_{a}(o)} \mathbb F.$
Suppose that
\begin{equation} \label{eqn_7.31a}
\min\limits_{B_{a}(o)}|\mathbb A|\neq0
\quad\text{and}\quad  \mathbb F ^{\half} (\mathbb K+\|\nabla \log |\mathbb A|\|^2_f+ \mathbf{\Psi})\leq 4
\quad\text{in }B_{a}(o).\end{equation}
Then,
\begin{equation}\label{eqn_7.32a} {\mathbb F ^{\half}} \mathbf{\Psi}
 \leq \mff C_{5}\left[\mathbb F _\diamond ^{\half}\max_{B_{a}(o)}|\mathbb A |
 +  \mathbb F _\diamond ^{\frac{1}{3}}\max_{B_{a}(o)}|\mathbb A |^{\frac{2}{3}}  +
a\inv \mathbb F _\diamond^{\frac{1}{4}}+ a^{-2}\mathbb F _\diamond ^{\half} \right]
\quad\text{in }
B_{a/2}(o),
 \end{equation}    where $\mff C_5$ is a constant.
 \end{lemma}

By (5), we know that, for any fixed $R$ and for any small constant $\epsilon>0$,
\begin{equation}\label{eqn_4.19}
1-\epsilon \leq  \check{ \mathbb F}_{k}(z)\leq 1+\epsilon,\;\;\;\;\;
 \check{ \mathbf \Psi}_{k} \leq \epsilon\quad\text{for any } z\in B_R(\check z^o_k),
\end{equation}
when $k$ is large enough. By (5) and (6), \eqref{eqn_4.19} also holds in $B_R(\check z^\ast_k).$

It follows from (4) and \eqref{eqn_4.19} that
\begin{equation}\label{eqn_4.20}
\check{\mc R}_{k}\leq 5  \quad\text{for any } z\in B_R(\check z^o_k).
\end{equation}
To derive a contradiction we need the  convergence of $\check f_k.$ We discuss two cases.

{\bf Case 1.} There is a constant $C'>0$ such that
$\check d\indexm(\check z^o\indexm,\check z^\ast\indexm)\geq C',$

{\bf Case 2.} $\lim\limits_{k\to\infty}\check d\indexm(\check z^o\indexm,\check z^\ast\indexm)=0.$

\v
By affine transformations as in  \eqref{eqn_4.16}
with $\lambda=1,$ we can  minimal-normalize $(\check u, \check  p^\circ, \check {\check p})$. To simply notations we still denote them by  $ \check
u,\check d$ and etc,   after this transformation.
   Since  $\check {\mathcal R},\check {\mathcal K},\check {\mathbb K},\tilde{\bf \Psi}$ and  $\check  W(z)/\check  W(z')$ for any $z,z'\in B_{N}(\tilde z^o)$  are invariant under these transformations,  then $(\check u, \check  p^\circ, \check {\check p})$ satisfies the assumption of Theorem \ref{theorem_7.4.1}.
 \v

 As in \cite{CLS-1},  for both cases, by Theorem \ref{theorem_7.4.1},  we have $\check z^o\indexm$ converges to a
point $\check z^\circ_\infty $ and $\check f_k$  $C^{3,\alpha}$-converges to a
function $\check f_\infty$  in a neighborhood of $D_{b_1}(\check z^\circ_\infty),$ and
\begin{equation}\label{eqn_4.25}
\check W\equiv const.,\;\;\;\;\;\;C_1\inv\leq \check
f_{i\bar j}\leq C_1\quad\text{in }D_{b_1}(\check z^\circ_\infty),
\end{equation}  where $C_1>0$ is a constant and $b_1$ is a constant independent of $k$. Here
\eqref{eqn_4.25}
follows from \eqref{eqn_4.19} and the $C^{3,\alpha}$-convergence of $\check f_k$.

By the same argument of \cite{CLS-1}, we check that
\begin{equation}\label{eqn_4.26}
\lim\limits_{k\to\infty}\max\limits_{D_{b_1}(\check z^\circ_\infty)}\|\nabla \log | \check{\mathbb A}_{k}|\|^2_{\check f_k}=0.
\end{equation}
Note that
$$
\frac{\p }{\p \check z^{i}} \log  \check {\mathbb  D}_{k}
=\sum\frac{\p \check \xi_{j}}{\p \check z^{i}} \frac{\p\log  \check {\mathbb  D}_{k}}{\p \check \xi_{j}}.
$$
Then, by \eqref{eqn_4.25}, we can check that
$$
\lim_{k\to\infty} \| \nabla \log  \check {\mathbb  D}_{k} \|=0,\;\;\;\;
 \lim_{k\to\infty} \|\nabla^2 \log  \check {\mathbb  D}_{k} \|=0.
$$
It follows that
\begin{equation}\label{eqn_4.27}
\lim_{k\to\infty}\mc {\check K}_k(\check z_k^o)=\lim_{k\to\infty}\mathbb {\check K}_k(\check z_k^o).
\end{equation}

Combining \eqref{eqn_4.19}, \eqref{eqn_4.26}, \eqref{eqn_4.27}  and
$\mc {\check R}_k(\check z_k^o)=1  $, we get, for $k$ large enough,
\begin{equation}
\label{eqn_4.28}2\mc {\check K}_k(\check z_k^o)\geq 1-\epsilon .\end{equation}
On the other hand,
we can conclude that $ \check{\mc K}\equiv 0 $ as in \cite{CLS-1}.   This contradicts
\eqref{eqn_4.28}.
The theorem is proved.
\end{proof}


\section{An Upper Bound of $H$}\label{sect_6}

In this section, we present
an upper bound of a function $H$ to be defined below.

We assume that $ \Delta\subset \mathbb R^{n}$ is a Delzant polytope and
that $M$ is the corresponding toric variety. Set
$$
\mc C^\infty(M,\omega_g)=\{f|f=g+\phi, \phi\in C^\infty(M) \mbox{ is } T^{n}\mbox{-invariant}\}.
$$
  Fix a large constant $K_o>0$. We set
$$
\mc C^\infty(M,\omega_g;K_o)
=\{f\in \mc C^\infty(M,\omega_o)|
|\mc S(f)|\leq K_o\}.
$$
Choose a local holomorphic coordinate system $z_1,...,z_n$ on $M$. Set
$$
H=\frac{\det(g_{i\bar{j}})}{\det(f_{i\bar{j}})}.$$
It is known that $H$ is a global function defined on $M$.
Set
$$\mathbb{H}=\frac{\mathbb{F}_g}{\mathbb{F}_f}=
\frac{\mathbb{D}_g}{\mathbb{D}_f}H,$$
where
$$\mathbb{D}_g:=(\nabla^g)^*\mathbb{D},\;\;\;\mathbb{D}_f:=(\nabla^f)^*\mathbb{D}.$$

Let $\mu:M\rightarrow \bar{\Delta}\subset \mathbb{R}^{n}$ be the moment map.
We introduce notations:
$$\aligned R_g&= \max\limits_{M}\left\{\left|\sum g^{ij}\left(\log \mathbb{F}_g\right)_{ij}\right|\right\},\\
\mathcal{D}&=\max\limits_{\bar{\Delta}}\left\{\left|\frac{\p }{\p \xi_j}(\log \mathbb D)\right|\right\},\\
\mathcal{R}&=\max\left\{ R_g, \mathcal{D}^2,diam(\Delta)\right\}.\endaligned$$
Li, Lian, and Sheng \cite{LLS-1} proved the following result.

\begin{theorem}\label{theorem 6.1}  For any $\phi\in C^{\infty}(M)$, there holds
\begin{equation}\label{eqn_3.2}
H \leq \mff C_{6} \exp
\left\{(2\mathcal{R}+1)\left(\max_M\{\phi\}-\min_M\{\phi\}\right)\right\},
\end{equation}
where $\mff C_{6}$ is a constant depending only on $n$, $\max|A|$ and $\mathcal R.$
\end{theorem}

\def \minequiv{\operatorname*{\sim}\limits^{min}}

\section{A Lower Bound of $H$}\label{sect_7}

In this section, we present
a lower bound of the function $H$ defined in the previous section.
We first introduce a subharmonic function.


Let $\vartheta$ be a vertex and $\cplane^2_\vartheta$ be a coordinate chart associated to the
vertex $\vartheta$. Now consider an element $\mathsf{f}\in \mc C^\infty(M,\omega )$.
Let $f_\vartheta$ be its restriction to  $\CHART_\vartheta$.
We introduce a function
$$
F_\vartheta=\log \mathbb F_\vartheta+Nf_\vartheta.
$$

We now prove that $F_\vartheta$ is a  {\em subharmonic}  function, with respect to $\square$,
the Laplacian operator of the metric $\mathcal G_{f}.$
In \cite{LSZ} we proved the following result.

\begin{lemma}\label{lemma_7.1.2}
Choose $\{o,\nu^{i},\;i=1,...,n\}$ as a base of $\mathbb{R}^{n}$,
let $\bar{\Delta}\subset \{(\xi_1,...,\xi_{n})| \xi_1>0,\;\xi_2>0,\;...,\xi_{n}>0\}$ be a Delzant polytope satisfying
\begin{equation}
 \sum_{\alpha\in R_{M^{+}}}\frac{\left(\sum_{j=1}^{n} M_{\alpha}^{j}\right)diam(\Delta)}{D_{\alpha}}<\frac{n}{4}.
\end{equation}
Then there is a constant $N>0$, depending only on $n$, $\mathbb{D}$, $\Delta$,
and the position of $\Delta$ in $\mathbb{R}^n$, such that for any vertex $p$ of $\Delta$
$$
\square (\log \mathbb F_{p}+Nf_{p})>0.
$$
\end{lemma}

\def \fkz{\mathfrak{z}}
\def \sff{\mathsf{f}}

Here $p=\vartheta$ and $n=2.$ Set
$$W_{g_\vartheta}=\det((g_\vartheta)_{i\bar j}),\;\;\; \mathbb F_{g}=W_{g_\vartheta} \mathbb D_{g}.$$
Since  $\mathbb F_{g}$ is uniformly bounded and
\begin{equation}\label{eqn_6.2}
C^{-1}\leq \mathbb D_{f}\leq C,\;\;C^{-1}\leq \mathbb D_{g}\leq C
 \end{equation}for some constant $C>0,$ by the same argument of \cite{CLS-2}, we have the following result.

\begin{theorem}\label{theorem_7.0.1}
Let $\Delta\subset \mathbb R^2$ be a Delzant polytope and
 $(M,\omega_o)$ be the associated  compact  toric surface. Assume that $\mathbb D$
 is an edge-nonconstant function and $u\indexm=v+\psi\indexm\in \mc
C^\infty(\Delta,v)$ be a sequence of functions satisfying \eqref{eqn_3.1}.
Suppose that
\begin{enumerate}\item[(1)] $A\indexm$ $C^{3}$-converges to $A$ on $\bar\Delta$;
\item[(2)]  $\max_{\bar \Delta}|u_k|  \leq \mff C_o$,\end{enumerate}
where $\mff C_o$ is a constant independent of $k$.
 Then, for any
$k$,
\begin{equation}\label{eqnz_7.1}
 \mff C_{7}\inv\leq H_{f_k}\leq \mff C_{7},\end{equation}
where $\mff C_{7}$ is a positive constant independent of $k$.
\end{theorem}

\section{A Convergence Theorem}\label{sect_Convergence}

In this section, we extend Theorem 6.11  \cite{CLS-1} to the generalized Abreu equations.

\begin{lemma}\label{lemma_4.1}
 Let $z_o\in\CHART_\bullet$, where $\bullet$ can be $\Delta, \ell_i$ or $\vartheta_i$, and
 $B_a(z_o)$ be a geodesic ball in $\CHART_\bullet$.
 Suppose that
$f(z_o)=0$, $\nabla f (z_o)=0$ and
 $$ \mathcal K(f) \leq \mff N_4,\;\;\;\;    \det(f_{i\bar j})\leq  \mff N_4,\;\;\; |z|\leq \mff N_4\quad\text{in }B_a(z_o),$$
for some positive
 constant $\mff N_4$.
 Then there is a constant $a_1>0$, depending on $a$ and $C_1$,
such that  $D_{{2a_1}}(z_o)\subset B_{{a}/{2}}(z_o),$ and for any $k\geq 0,$
$$
\|f\|_{C^{k+3,\alpha}(D_{a_1}(z_o ))}\leq \mff C_{8}(a,\mff N_{4}, \|A\|_{C^k }, \|\mathbb D\|_{C^{k+1}}),
$$
where $\mff C_{8}>0 $ is a constant depending only on $a$, $\|A\|_{C^k },$ $\|\mathbb D\|_{C^{k+1}}$ and $\mff N_{4}$.
\end{lemma}

\begin{proof}   We only prove the case $B_a(z_o)\subset \CHART_\ell$;
the proof for other cases is similar. Assume that $Z_{\ell}=\{z_{1}=0\}.$  As in \cite{CLS-1}, we can prove
$$  C_{1}\inv\leq \lambda_1\leq  \lambda_2\leq   C_{1}\quad\text{for any }q\in B_{a/2}(p_0),$$
and
\begin{equation} \label{eqn_2.10}
\|f\|_{C^{2}}\leq  C_{1},
\end{equation}
where $\lambda_1, \lambda_2$ are eigenvalues of the matrix $(f_{i\b
 j})$ and $  C_{1}$
 is a positive constant depending on $n,a$ and $\mff N_4.$
Note that $\mathbb F_{\ell}=W\mathbb D,$ where $W=\det(f_{i\bar j}).$
 In coordinates $\mathsf U_{\ell}$ we can rewrite \eqref{eqn_4.3} as
\begin{equation}\label{complex_scalar_curvature_equations}
-\sum f^{i \bar j}\frac{\p^2 (\log  ( W\mathbb D))}{\p z_{i}\p \bar z_{j}}
+\frac{1}{2}\sum \frac{\p \log \mathbb D}{\p \xi_{i}}Re\left[ z_{i}\frac{\p \log  ( W\mathbb D)}{\p z_{i}}
 \right]=\mathbb A_{\ell}.\end{equation}
We point out that first-order terms appear in \eqref{complex_scalar_curvature_equations}
due to the presence of $\mathbb D$ for homogeneous toric bundles.
By applying Krylov-Safonov's estimate to the equation \eqref{complex_scalar_curvature_equations},
 we have $\log (W\mathbb D)\in C^{\alpha}(U)$, for some $\alpha\in (0,1)$.  By $$
\xi_1= {z_1} \frac{\partial f}{\partial z_1},\;\;\;\;
\xi_2=2\frac{\partial f}{\partial w_2},
 \;\;\;\;\frac{\p \log\mathbb D}{\p z_{i}}=\frac{\p \xi_{j}}{\p z_{i}}\frac{\p  \log\mathbb D}{\p x_{j}}$$
 and \eqref{eqn_2.10},
we can check  that
$
\log  \mathbb D\in C^{1}(U).
$
It follows that
$$
  \det(f_{i\bar j})\in C^{\alpha}(U).
$$
Then by the same argument as in \cite{CLS-1}, we obtain the desired estimate.
\end{proof}


\section{The Proof of Main Theorem}\label{sec_8}

In this section, we prove the existence of
the solution to \eqref{eqn 2.7} by the standard continuity method.
The closedness is provided by Theorem \ref{theorem_2.2.2} below.
For technical reasons,
we are only  able to prove this under the condition that $\mathbb A_{\ell}$ is an edge-nonvanishing.
Due to this, we need to require that the Duistermaat-Heckman polynomial
$\mathbb D$ is an  edge-nonconstant
function.

\subsection{The Method of Continuity}\label{sec_MethodContinuity} In this subsection,
we construct a 1-parameter family of equations to solve \eqref{eqn 2.7} by the method of continuity.
We first demonstrate that edge-nonconstant $\mathbb D$ implies
edge-nonvanishing $\mathbb A_{\ell}$.

\begin{remark}\label{remark_9.1}
Let $\ell\subset \p \Delta$ be an edge and $(\xi^{\ell}_{1},\xi^{\ell}_{2})$
be the coordinates such that $$\ell\subset \{\xi^{\ell}|\xi^{\ell}_{1}= 0\},\;\;\;\;
\Delta\subset \{\xi^{\ell}|\xi^{\ell}_{1}\geq 0\}.$$
Then,
$$
\mathbb A_{\ell}=A+\sum_{k=1}^{2} \frac{\p \log \mathbb D}{\p \xi_{k}^{\ell}}\frac{\p x^{1}_{\ell}}{\p x^{k}_{\ell}}.
$$
We consider the coordinate transformation
$$
\tilde \xi_{1}=\xi^{\ell}_{1},\; \tilde \xi_{2}=\xi^{\ell}_{1}+a\xi^{\ell}_{2},
$$ where $a$ is a constant to be determined.
By the Legendre transforation we have
$$\tilde x^{1}=x^{1}_{\ell}-ax^{2}_{\ell},\; \tilde x^{2}=x^{2}_{\ell}.$$
Set
$$\tilde {\mathbb F}=\det\left(\frac{\p^2 f}{\p\tilde x^{i} \p\tilde x^{j}}\right)e^{- \tilde x^{1}},\quad
\tilde {\mathbb A}=-\square \log \tilde {\mathbb F}.$$ Then,
\begin{equation}\label{eqn_9.1}
\tilde {\mathbb A}=A+\sum_{k=1}^{2} \frac{\p \log \mathbb D}{\p \xi_{k}^{\ell}}\frac{\p \tilde x^{1}}{\p x^{k}_{\ell}}
=\mathbb A_{\ell}-a\frac{\p \log \mathbb D}{\p \xi_{2}^{\ell}}.
\end{equation}
Since $\mathbb D$ is an edge-nonconstant function, we have  $\frac{\log \mathbb D}{\p \xi^{2}_{\ell}}\neq 0.$
By choosing $a$ appropriately, we conclude that $\tilde {\mathbb A}|_{\ell}\neq 0$.
Set $$A_{0}=\mathcal S(v)+h_{G},\;\;\; \mathbb A_{\ell}^{0}=A_{0} -a\frac{\p \log \mathbb D}{\p \xi_{2}^{\ell}},\;\;\;A_{1}=A ,\;\;\;
\mathbb A_{\ell}^{1}=A_{1} -a\frac{\p \log \mathbb D}{\p \xi_{2}^{\ell}}.$$
 Hence by choosing $a$ in \eqref{eqn_9.1} appropriately,
 we can fix $q_{\ell} \in \ell$ and  coordinates
 $(\xi^{\ell}_{1},\xi^{\ell}_{2})$ for any edge of $\p \Delta,$ such that
\begin{equation}\label{eqn_9.2}
\xi^{\ell}(q_{\ell})=0,\;\;\;\; \;\;\ell\subset \{\xi^{\ell}|\xi^{\ell}_{1}= 0\},\;\;\;\;\Delta\subset \{\xi^{\ell}|\xi^{\ell}_{1}\geq 0\},
\end{equation}
and
\begin{equation}
|\mathbb A_{\ell} (q_{\ell})|>2\delta_{o}>0,\;\;\;|\mathbb A^{o}_{\ell} (q_{\ell})|>2\delta_{o}>0,\;\;\;
{\mathbb A}_{\ell}^{o}\mathbb A_{\ell}>0,
\end{equation}
for some constant $\delta_{o}>0.$ Hence, we can assume that the whole path $A_t$,
connecting $A_0$ and $A_1=A$, such that
${\mathbb A}_{\ell}^{t}=(1-t){\mathbb A}_{\ell}^{0}+t {\mathbb A}_{\ell}^{1}$ satisfies  \eqref{eqn_2.2} below.
On each $\ell$, let $\epsilon$ be  a positive constant such that
\begin{equation}\label{eqn_2.1}
D_{2\epsilon}(q_{\ell})\cap \partial \Delta\subset \ell,
\end{equation}
and, for any $t\in [0,1],$
\begin{equation}\label{eqn_2.2}
|\mathbb A^{t}_{\ell}|>\delta_o \quad\text{on } \mc D^\ell:= D_\epsilon(q_{\ell})\cap  \bar\Delta,
\end{equation}
for a constant $\delta_o>0$.
\end{remark}

Let $K$ be a scalar function on $\bar\Delta$ of $G\times_{T}M$ and suppose that there exists a constant
$\lambda>0$ such that $\Delta$ is $(\mathbb{D},A,\lambda)$ stable, where $A=K-h_{G}$.

Let $I=[0,1]$ be the unit interval. At $t=0$ we start with a known
metric $\mathcal G_{v}$.
Let $K_0$ be its scalar curvature
on $\Delta$ and $A_{0}=K_{0}-h_{G}$. Then, $\Delta$ is $(\mathbb{D}, A_0, \lambda_0)$ stable for some constant
$\lambda_0>0$ (cf. \cite{LLS}). At $t=1$, set
$(A_1,\lambda_1)=(A,\lambda)$. On $\Delta,$ set
$$
A_t=tA_1+(1-t)A_0, \;\;\; \lambda_t=t\lambda_1+(1-t)\lambda_0.
$$
Obviously, $
A_t=tA_1+(1-t)A_0$.
It is easy to verify that $\Delta$ is $(\mathbb{D},A_t,\lambda_t)$ stable.

\begin{remark}\label{remarka_2.2.1}
 For any $t\in [0,1],$ $\Delta$ is $(\mathbb{D},A_t,\lambda')$-stable,
 where $\lambda'=\min\{\lambda_{0} ,\lambda_1\}.$
\end{remark}

Let $u_t\in \mathbf S_{p_{o}}$ be a solution of the equation
$$-\frac{1}{\mathbb D}\frac{\p ^2 \mathbb D u_{t}^{ij}}{\p \xi_{i} \p \xi_{j}}=A_t.$$
Applying $C_{0}$ estimates, we have
$$|\max\limits_{\bar \Delta} u_t-\min\limits_{\bar \Delta} u_t|\leq \mathcal C_1
\quad\text{for any } t\in [0,1],$$
where $\mathcal C_1$ is a constant depending only on  $\lambda'$, $\Delta$, $\|\mathbb D\|_{C^{0}}$ and
$\|A_0\|_{C^{0}}+\|A_1\|_{C^{0}}.$

Set
$$
\Lambda=\left\{t|-\frac{1}{\mathbb D}\frac{\p^2 \mathbb D u^{ij}}{\p \xi^{i} \p \xi^{j}}= A_t
\text{ has a solution in }  \mathbf S\right\}.
$$
We will prove that $\Lambda$ is open and closed. The openness is
standard by following an argument by  Lebrun and Simanca \cite{L-Ss}.
The closedness is provided by Theorem \ref{theorem_2.2.2} below.
This then implies Theorem
\ref{theorem_1.1}.

\subsection{Regularity near $\partial\Delta$}\label{sect_Regularity}
In this subsection, we discuss the regularity near $\partial\Delta$.
The following result can be regarded as the
main result in this paper.

\def \fkz{\mathfrak{z}}

\begin{theorem}\label{theorem_2.2.2}
Let $\Delta\subset \mathbb R^2$ be a Delzant polytope and
 $(M,\omega_o)$ be the associated  compact  toric surface.
 Suppose that $\mathbb D$ is an edge-nonconstant   function on $\bar\Delta$. Let $u\indexm=v+\psi\indexm\in \mc
C^\infty(\Delta,v)$ be a sequence of functions satisfying \eqref{eqn_3.1}. Suppose that
\begin{enumerate}\item[(1)] $A\indexm$ $C^{3}$-converges to $A$  on $\bar\Delta$;
\item[(2)]  $
\max_{\bar \Delta}\left|u_k\right| \leq \mff C_o$,\end{enumerate}
where $\mff C_o>0$ is a constant independent of $k$.
Then,  there is a subsequence of  $\psi\indexm$ which  $C^{6,\alpha}$-converges to
a function $\psi\in  C^{6,\alpha}(\bar \Delta)$ with $\mathbb
S(v+\psi)=K$ where $K=A+h_{G}$.
\end{theorem}


\begin{proof}
We estimate $\psi\indexm$ in two steps, near edges in the first step and near vertices
in the second.

{\em Step 1: Estimates on edges.}
We will prove the  regularity near each edge $\ell$ where $\mathbb A_{\ell}$ does not vanish.
Let $\ell$ be any edge and $(\xi^{\ell}_{1},\xi^{\ell}_{2})$ be the coordinates such that
$$
\xi^{\ell}_{1}(\ell)=0,\;\;\;\Delta\subset \{\xi^{\ell}|\xi^{\ell}_{1}\geq 0\}.
$$
By Remark \ref{remark_9.1}, we can assume that  $q_{\ell}\in \ell$ such that $|\mathbb A_{\ell}(q_{\ell})| >0$.
Recall that $f_{k}=g +\phi_k,$ where $f_k,g$ are Legendre transform of $u_k,v$ respectively
and $\phi_k\in C^\infty_{\mathbb T^2}(M).$
Set
$$\Omega=\{(z_1,z_2)|\log|z_1|^2\leq \half ,|\log |z_2|^2|\leq 1\}.$$
By the convergence of $A_k$ to $A$,
we can assume
that
$$\aligned &|{\mathbb A}_k|>\delta >0\quad\text{in }\Omega,\\
&D_{2a}(\xi^{(\ell)})\subset \tau_{f_k} (\Omega),\endaligned$$
and
$$|z_1(\fkz^{(\ell)}_k) |=0,\quad |z_2(\fkz^{(\ell)}_k)|=1,$$
where $\delta,a$ are positive  constants independent of $k,$ and
  $\fkz^{(\ell)}_k\in Z_\ell$  whose image of the moment map is $\xi^{(\ell)}$.

  We omit the index $k$ if
there is  no danger of confusion. By Theorem \ref{theorem_5.3.3}, we conclude that
there is a constant $\epsilon>0 $ that is independent of
$k$   such that
$$B _\epsilon(\xi^{(\ell)})\cap \Delta\subset D_a(\xi^{(\ell)})\cap \Delta.$$ Then
$B _\epsilon(\fkz^{(\ell)})$ is uniformly bounded. Hence, on this domain, we
assume that all data of ${g_\ell}$
are bounded.

Note that  $\mathbb F_{g_\ell},W_{g_\ell}, \mathbb D_{g}$  and $\mathbb D_{f}$
are bounded on $B _\epsilon(\fkz^{(\ell)})$. By  Theorem \ref{theorem_7.0.1}, we have
\begin{equation}\label{eqn_8.6}
C_1\inv\leq W_{f} \leq C_1,\;\;\;\;C_1\inv\leq\mathbb F_{f} \leq C_1
\quad\text{in }B _\epsilon(\fkz^{(\ell)}).
\end{equation}
It follows from Lemma \ref{lemma_5.2.1} that
$$ \partial ^{2}_{22} h|_{D_a(\xi^{(\ell)})\cap \ell}\geq \mathsf C_{3}.$$
Then by Theorem \ref{theorem_5.0.1} and \eqref{eqn_8.6},
we conclude that there is a constant $C_2>0$ independent of $k$ such that
\begin{equation}\label{eqn_8.7}
\mc K+\mathbb K \leq C_2\quad\text{in }B _\epsilon(\fkz^{(\ell)}).
\end{equation}
By the convexity of $u$ and $\|u-v\|_{L^{\infty}(\Delta)}\leq \mathcal C_{1},$ we have
$$ |\partial _{2}u|\leq  \mff C_{o} a^{-1},\;\;\; \;\;\; \;\;\; \partial _{1} u\leq  \mff C_{o}a^{-1}.$$
That is
$\max_{B _\epsilon(\fkz^{(\ell)})}|z|\leq C_3.$
 Hence, by Lemma \ref{lemma_4.1},
 we have the regularity of $\mathsf{f} $
on $B_\epsilon(\fkz^{(\ell)})$.

{\em Step 2: Estimates on vertices.} We will prove
the regularity in neighborhoods of each vertex with the help of subharmonic functions in
Lemma \ref{lemma_8.2.1} below.

Let  $\mathbf {\phi}=\mathbf f-\mathbf g\in C^\infty(G\times_{K}M),$
where $\mathbf f$(resp. $\mathbf g$) are potential function of the metric $\mathcal G_{u}$(resp. $\mathcal G_{v}$).
Let $(t^{1},\cdots, t^{n+l})$ be the local holomorphic coordinates of $G\times_{K}M.$
Let $\mathbf f_{A\bar B}=\frac{\p^2 \mathbf f}{\p t^{A}\p \bar t^{B}}$ and $\mathbf f^{A\bar B}$
denotes the inverse of the matrix  $\mathbf f_{A\bar B}$.
Set
$$
T=\sum \mathbf f^{A\bar B}\mathbf g_{A\bar B}=n-\square \phi,\;\;\;
P=\exp(\kappa \mathbb F^\alpha)\sqrt{\mathbb F}\mathbf{\Psi},
\;\;\;Q=e^{- N_1(\phi-inf {\phi}+1)}\sqrt{\mathbb F}T.
$$
In the basis of $\{S_{j},\bar S_{j},S_{\alpha},\bar S_{\alpha}\}_{j\leq n,\alpha\in R_{M^{+}}},$
one can check that (see \cite{LSZ})
$$
\mathcal G_{f}=2\sum_{A,B=1}^{2+l} f_{,A\bar B}(\omega^{A}\otimes \bar\omega^{B})
$$
where $\{\omega^{A}\}$ are dual $(1,0)$ form of $\{S_{j}\} $, and $f$ is the function defined in \S\ref{sec-2.2}.
Then, $f$ is potential function of $G\times_{K}M$.
Hence,
$$
\phi=f-g.
$$ Note that $f_{\vartheta}$ and $f_{\ell}$ are not a potential function.
A direct calculation gives us
\begin{equation}\label{eqn_8.1}
T=\sum_{i=1}^{2} f^{ij}g_{ij}+\sum_{\alpha \in R_{M^{+}}}\frac{\sum\limits_{i,j=1}^{2} f^{ij}
\frac{\p D_{\alpha}}{\p x^{i}} \frac{\p g}{\p x^{j}}}{\sum\limits_{k,l=1}^{2} f^{kl}\frac{\p D_{\alpha}}{\p x^{k}}
\frac{\p f}{\p x^{l}}}\geq \sum_{i=1}^{2} f^{ij}g_{ij},
\end{equation}
where we used
$\frac{\p f}{\p x^{j}}>0,\;\frac{\p g}{\p x^{j}}>0,\;\sum\limits_{i=1}^{2} f^{ij}\frac{\p D_{\alpha}}{\p x^{i}}
=\frac{\p D_{\alpha}}{\p \xi_{j}}>0$ for any $ {j\leq 2,\alpha\in R_{M^{+}}}.$

Since $\sum\limits_{k,l=1}^{2} f^{kl}\frac{\p \log D_{\alpha}}{\p x^{k}} \frac{\p f}{\p x^{l}}=1$, we can check that
\begin{equation}\label{eqn_8.2}
T\leq \sum_{i,j=1}^{2} f^{ij}g_{ij}+\sum_{\alpha \in R_{M^{+}}} \sum_{i,j=1}^{2} f^{ij}
\frac{\p \log D_{\alpha}}{\p x^{i}} \frac{\p g}{\p x^{j}}\leq \sum_{i,j=1}^{2} f^{ij}g_{ij}+C,
\end{equation}
where $C$ is a constant depending only on $diam(D)$ and  $\max_{\Delta} |\nabla\log \mathbb D|.$
\v
Li, Sheng, and Zhao \cite{LSZ} prove the following lemma.
\begin{lemma}\label{lemma_8.2.1}
 Let
  $\vartheta$ be a vertex, $\Omega\subset \CHART_\vartheta$ and
  $ -\square  \log \mathbb F_{\vartheta} =\mathbb A_{\vartheta}$. Suppose that
\begin{equation}\label{eqn_8.3}
  \|\mathbb A_{\vartheta}\|_{C^{1}(\tau_f(\bar \Omega))} \leq  \mff N_5,\quad \mathbb F_{\vartheta}\leq \mff N_5,\quad
\sup|\phi|+\sup_{\Omega}|z|\leq \mff N_5\quad,\end{equation}
for some constant $ \mff N_5>0$ independent of $k$. Assume that \eqref{equ_R_3.6} holds with $n=2.$ Take
\begin{equation}\label{eqn_8.4} N_1=100,\alpha=\frac{1}{3},\kappa=[4\mff N_5^\frac{1}{3}]\inv.
\end{equation}
Then,
\begin{equation}\label{eqn_8.5}
\df (P+Q+ \mff C_{9}f_{\vartheta} ) \geq \mff C_{10}(P+Q)^2>0,
\end{equation} for some positive constants $\mff C_{9}$ and $\mff C_{10}$
depending only on $\mff N_2$, the structure constants of $\mathfrak{g}$, $\mathbb{D}$, $\Delta$ and the position of $\Delta$ in $\mathbb{R}^2$.
\end{lemma}

By Step 1, there is a bounded open set $\Omega_\vartheta\subset \CHART_\vartheta$, independent
of $k$,
such that $\vartheta \in \tau(\Omega_\vartheta)$ and the regularity of
$f_\vartheta$ holds in a neighborhood of $\partial \Omega_\vartheta$.
By the $G$-action, we get a $G$-invariant neighborhood in $G\times_KM$,
denotes by $G(\Omega_\vartheta).$ 
By \eqref{eqn_8.2} we have $T$ is uniform bounded in a neighborhood of $G(\p\Omega_\vartheta) $,
the boundary of $G(\Omega_\vartheta)$.
It follows from Lemma \ref{lemma_8.2.1} that $T$ is uniform bounded in  $G(\Omega_\vartheta)$.
We omit the index $k$ if
no confusion occurs.

By Lemma \ref{lemma_8.2.1}, $P$ and $Q$ are bounded above.
Since  $\mathbb F_{g_{\vartheta}},$ $\frac{\mathbb D_{g}}{\mathbb D_{f}}$
are uniformly bounded, by Theorem \ref{theorem_7.0.1},
we conclude that  $\mathbb F_{f_{\vartheta}}$ and $W$ is bounded below and above in $G(\Omega_\vartheta)$.
Then, $T$ is bounded above. Since $T$ is $G$-invariant,  by \eqref{eqn_8.1}  we have a constant $C_{1}>0$
such that
$$C_1\inv \leq \chi_1 \leq   \chi_2 \leq
C_1, \;\;\;\;\;\|\nabla \log\mathbb F_{f_\vartheta}\|_{\mathcal G_{f}}\leq C_{1}
$$ where  $\chi_1,\chi_2$ be the
eigenvalues of the matrix $(\sum g^{i \b j}f_{k\b j}).$  Let $\lambda_1,\lambda_2$
(resp. $\mu_1,\mu_2 $)  be the eigenvalue of the matrix  $(f_{i \bar j})$ (resp. the matrix $(g_{i\bar j})$).
Since $C_{2}^{-1}\leq  \mu_1, \mu_2\leq C_{2}$ for some constant $C_{2}>0,$  we can conclude
$$
C_{3}\inv\leq \lambda_1\leq\lambda_2\leq C_{3},\;\;\;\;\;\;\|\nabla
\mathbb F_{f_\vartheta} \|_{C^{1}(\Omega_\vartheta)}\leq C_{3} \;\;\;\;\mbox{ in } \Omega_\vartheta.
$$
By the  bound  of   $\|\nabla\log\mathbb D_{f}\|$,  we have
$$
\|\log  W_f\|_{C^{1}}\leq C_{4}.
$$

By Theorem 6.7 in \cite{CLS-1}, we have, for any $U\subset\subset \Omega_\vartheta$,
$$\|f_\vartheta\|_{C^{6,\alpha}(U)}\leq C_{5}.$$
where $C_{5}$ is a constant depends on $\|A\|_{C^{3}(\Delta)},$ $d_E(U,\Omega_\vartheta)$ and the bound of $\Omega_\vartheta.$

Since $f_{\vartheta}$ is $G$-invariant,  we get the interior regularity $f_{\vartheta}$ in
$G(\Omega_\vartheta)$; namely,
$\phi\indexm$ uniformly $C^{6,\alpha}$-converges to
a function $\phi\in  C^{6,\alpha} (G\times_{K}M)$ with $\mathbb
S(\phi+g)=K\circ \nabla^f$.
Then, $\psi\in C^{6,\alpha}(\bar \Delta) $ satisfies \eqref{eqn_3.1}.
\end{proof}

\end{document}